\documentclass[twoside,11pt]{amsart}
\usepackage{amsmath, amssymb, amsthm}
\usepackage{latexsym}
\usepackage{eucal,enumerate}
\usepackage{graphicx}
\usepackage{hyperref}

\swapnumbers
\theoremstyle{plain}
\newtheorem{theorem}{Theorem}[section]
\newtheorem{proposition}[theorem]{Proposition}
\newtheorem{lemma}[theorem]{Lemma}
\newtheorem{corollary}[theorem]{Corollary}
\theoremstyle{definition}
\newtheorem{definition}[theorem]{Definition}
\newtheorem{remark}[theorem]{Remark}
\newtheorem*{remark*}{Remark}
\newtheorem{fact}[theorem]{Fact}
\newtheorem{example}[theorem]{Example}
\newtheorem{notation}[theorem]{Notation}
\DeclareMathOperator{\Mod}{Mod}
\DeclareMathOperator{\acl}{acl}
\DeclareMathOperator{\dcl}{dcl}
\DeclareMathOperator{\Aut}{Aut}
\DeclareMathOperator{\tp}{tp}
\makeatletter
\def\dotminussym#1#2{%
  \setbox0=\hbox{$\m@th#1-$}%
  \kern.5\wd0%
  \hbox to 0pt{\hss\hbox{$\m@th#1-$}\hss}%
  \raise.3\ht0\hbox to 0pt{\hss$\m@th#1\cdot$\hss}%
  \kern.5\wd0}
\newcommand{\dotminus}{\mathbin{\mathpalette\dotminussym{}}}
\makeatother

\makeatletter
\def\indsym#1#2{%
 \setbox0=\hbox{$\m@th#1x$}%
 \kern\wd0%
 \hbox to 0pt{\hss$\m@th#1\mid$\hbox to 0pt{$\m@th#1^{#2}$\hss}\hss}%
 \lower.9\ht0\hbox to 0pt{\hss$\m@th#1\smile$\hss}%
 \kern\wd0}
\newcommand{\ind}[1][]{\mathop{\mathpalette\indsym{#1}}}
\def\nindsym#1#2{%
 \setbox0=\hbox{$\m@th#1x$}%
 \kern\wd0%
 \hbox to 0pt{\hss$\m@th#1\not$\kern1.4\wd0\hss}
 \hbox to 0pt{\hss$\m@th#1\mid$\hbox to 0pt{$\m@th#1^{#2}$\hss}\hss}%
 \lower.9\ht0\hbox to 0pt{\hss$\m@th#1\smile$\hss}%
 \kern\wd0}

\makeatother
\renewcommand{\phi}{\varphi}
\newcommand {\N}{\mathbb{N}}   
\newcommand {\R}{\mathbb{R}}   
\newcommand{\U}{\mathbb{U}}

\newcommand{\B}{\mathcal{B}}
\newcommand{\M}{\mathcal{M}}
\newcommand{\n}{\mathcal{N}}     
\newcommand{\W}{\mathcal{W}}
\DeclareMathOperator{\dist}{dist}		     
\newcommand{\RT}{\mathbb{R}\mathrm{T}}  
\newcommand{\rbRT}{\mathrm{r}\mathrm{b}\mathbb{R}\mathrm{T}}  
\newcommand{\HM}{\mathrm{H}\mathrm{M}} 
\title[model theory of $\R$-trees]
{Model theory of $\R$-trees}

\author[S.~Carlisle]
{Sylvia Carlisle}

\address{Sylvia Carlisle \newline
Rose-Hulman Institute of Technology \newline
Terre Haute, Indiana 47803, USA}
\email{carlisle@rose-hulman.edu}

\author[C.~W.~Henson]{C.~Ward Henson}

\address{C.~Ward Henson \newline
University of Illinois at Urbana-Champaign \newline
Urbana, Illinois 61801, USA}
\urladdr{https://faculty.math.illinois.edu/\textasciitilde henson}

\thanks{Research supported by Simons Foundation grants (\#202251 and \#422088) to the second author.} 

\thanks{Revised February 2, 2021.}  

\begin{document}

\begin{abstract}
We show the theory of pointed $\R$-trees with radius at most $r$ is axiomatizable in a 
suitable continuous signature.  We identify the model companion $\rbRT_r$ of this theory 
and study its properties.  In particular, the model companion is complete and has quantifier 
elimination; it is stable but not superstable.  We identify its independence relation and find 
built-in canonical bases for non-algebraic types. Among the models of $\rbRT_r$ are 
$\R$-trees that arise naturally in geometric group theory.  In every infinite cardinal, we 
construct the maximum possible number of pairwise non-isomorphic models of $\rbRT_r$; 
indeed, the models we construct are pairwise non-homeomorphic.  We give detailed 
information about the type spaces of $\rbRT_r$.  Among other things, we show that the 
space of $2$-types over the empty set is nonseparable.  Also, we characterize the principal 
types of finite tuples (over the empty set) and use this information to conclude that 
$\rbRT_r$ has no atomic model.
\end{abstract}

\maketitle





\section{Introduction}

Continuous logic is an extension of classical first order logic used to study the model theory of 
structures based on metric spaces.  In this paper, we use continuous logic as presented in 
Ben Yaacov, Berenstein, Henson and Usvyatsov \cite{N} and 
Ben Yaacov and Usvyatsov \cite{BU} to study the model theory of $\R$-trees.  

An $\R$-tree is a metric space $T$ such that for any two points $a,b \in T$ there is a unique 
arc in $T$ from $a$ to $b$, and that arc is a geodesic segment ({ie}, an isometric copy of some closed interval in 
$\R$).  These spaces arise naturally in geometric group theory, for example: the 
asymptotic cone of a hyperbolic finitely generated group is an $\R$-tree.

An $\R$-tree may be unbounded, while the existing full treatments of continuous model theory 
are restricted to bounded structures. With this in mind, we consider pointed trees, choose a
real number $r>0$, and axiomatize the theory $\RT_r$ of pointed $\R$-trees of
radius at most $r$ in a suitable continuous signature. 

We then define the notion of \textit{richly branching} and axiomatize the theory
$\rbRT_r$ of the class of richly branching pointed $\R$-trees with radius $r$.  We prove that the 
models of $\rbRT_r$ are exactly the existentially closed models of $\RT_r$; thus $\rbRT_r$ is 
the model companion of $\RT_r$.  Next, we investigate some model theoretic properties of 
$\rbRT_r$, showing that it is complete and has quantifier elimination.  In 
particular, that means $\rbRT_r$ is the model completion
of $\RT_r$.  Further, we prove that $\rbRT_r$ is stable but not superstable and identify its 
model-theoretic independence relation. We characterize the principal types of $\rbRT_r$, 
and show that this theory has no atomic model. Finally, we show that $\rbRT_r$ is highly 
non-categorical.  In fact, for any density character this theory has the maximum possible 
number of pairwise non-isomorphic models; indeed, the models we construct are pairwise 
non-homeomorphic. We also give examples of richly branching $\R$-trees which come from 
the literature, including some that will be familiar to geometric group theorists.

In the remainder of this introduction we detail the contents of each section of this paper:

In Sections 2 and 3 we provide background concerning $\R$-trees and continuous logic, respectively.  
In Section 4 we specify a continuous signature $L$ suitable for the class of pointed $\R$-trees of radius 
at most $r$, and axiomatize this class of $L$-structures; the theory of the class is denoted $\RT_r$.

In Section 5 we discuss definability of certain sets and functions in $\RT_r$.  In Section 6 we show 
$\RT_r$ has amalgamation over substructures.  This plays an important role in many of the primary results in this paper.

In Section 7 we introduce the class of richly branching pointed $\R$-trees with radius $r$ and 
axiomatize this class.  The associated theory is denoted $\rbRT_r$.  We then show that $\rbRT_r$ 
is the model companion of $\RT_r$.  The theory $\rbRT_r$ is the main object of study in this paper.

In Sections 8 and 9 we verify the main model-theoretic properties of $\rbRT_r$.  We show that 
this theory is complete and admits quantifier elimination.  We characterize its types over sets of 
parameters and use this to show $\rbRT_r$ is $\kappa$-stable if and only if $\kappa = \kappa^\omega$; 
hence this theory is strictly stable (stable but not superstable).  We also show that $\rbRT_r$ is not a 
small theory; indeed, its space of $2$-types over $\emptyset$ has metric density character $2^\omega$.  
(The space of $1$-types over $\emptyset$ is isometric to the real interval $[0,r]$ with the usual 
absolute value metric.)  We give a simple geometric characterization of the independence relation of 
$\rbRT_r$.  Finally, we show that non-algebraic types have built-in canonical bases ({ie}, these 
bases are sets of ordinary elements in models of $\rbRT_r$ and do not require the introduction of  imaginaries).

In Section 10 we discuss some models of $\rbRT_r$ that have been constructed within the theory of $\R$-trees 
by Dyubina and Polterovich \cite{DP1,DP2} and some other models that arise in geometric group theory.

In Section 11  we use amalgamation constructions to build large families of models of $\rbRT_r$ 
and to characterize its isolated types over $\emptyset$.  
For each infinite cardinal $\kappa$, we show there are  $2^\kappa$-many pairwise non-isomorphic 
models of $\rbRT_r$ of density character $\kappa$.  This is the maximum possible number of models, 
and the models we construct are, in fact, pairwise non-homeomorphic. We show that $\rbRT_r$ 
has very few isolated $n$-types over $\emptyset$ and conclude that it has no atomic model 
(equivalently, it has no prime model). 

In Section 12 we briefly discuss how the results in this paper could be obtained
for the full class of pointed $\R$-trees ({ie}, without imposing a boundedness requirement).

\section{\texorpdfstring{$\R$}{R}-trees}
In this section we give some background concerning $\R$-trees.

\begin{definition}
An \emph{arc} in a metric space $M$ is the image of a homeomorphism $\gamma$ from an
interval $[0,r] \subseteq \R$ into $M$  for some $r \geq 0$.
A \emph{geodesic segment} in a metric space $M$ is the image of an
isometric embedding $\gamma \colon [0,r] \to M$.
We say that such an arc or geodesic segment is \emph{from $\gamma(0)$ to $\gamma(r)$}.  
A metric space $M$ is called \emph{geodesic} if for every $a,b \in M$ there is at least one
geodesic segment in $M$ from $a$ to $b$.
\end{definition}

\begin{fact}
\label{midptfact}
A complete metric space $M$ is a geodesic space if and only if for any $x,y\in M$ there
exists a midpoint $z$ between $x$ and $y$. That is, there exists $z$ such that:
\[
d(x,z)=\frac{d(x,y)}{2}, \text{ and } d(y,z)=\frac{d(x,y)}{2} \text{.}
\]
\end{fact}
\begin{proof}
See Bridson and Haefliger \cite[I.1.4, page 4]{BH}.
\end{proof}

\begin{definition}
An {\it $\R$-tree} is a metric space $M$ such that for any two
points $a,b \in M$ there is a unique arc from $a$ to $b$, and that arc is a geodesic
segment.
\end{definition}

In an $\R$-tree, $[a,b]$ denotes the unique geodesic segment from $a$ to $b$.
Since metric structures are required to be based on complete metric spaces,
it is a helpful fact that the completion of an $\R$-tree is an $\R$-tree
(see Chiswell \cite[Lemma 2.4.14]{C}).

Let $M$ be an $\R$-tree and $a\in M$.  Call the connected components of
$M\setminus\{a\}$ {\it branches} at $a$.  Let the \emph{degree} of a point
$a\in M$ be the cardinal number of branches at $a$.  If there are three or more
branches at $a\in M$, then we call $a$ a \emph{branch point}.
The {\it height} of a branch $\beta$ at $a$ is $\sup\{d(a,x)|x\in \beta\}$ if that
supremum exists, and is $\infty$ otherwise.  A \emph{subtree} of $M$ is any
subspace of $M$ that is itself an $\R$-tree.
A \emph{ray} in an $\R$-tree is an isometric copy of $\R^{\geq 0}$.  If $a\in M$,
a \emph{ray at $a$} is a ray so that the image of $0$ under the isometric embedding of $\R^{\geq0}$
into $M$ is $a$.

The following lemmas and definitions collect some straightforward facts about $\R$-trees
used in this paper. For helpful pictures and more facts about $\R$-trees see \cite{C}.

\begin{lemma}
\label{branch}
If $M$ is an $\R$-tree and $a,b,c \in M$, then
\begin{enumerate}
\item $d(a,b)+d(b,c)=d(a,c)+2 \dist(b,[a,c]).$
\item $b\in [a,c]$ if and only if $d(a,c)=d(a,b)+d(b,c)$.
\item For $b$ distinct from $a$ and $c$, we have $b\in [a,c]$ if and only if $a$ and $c$ are in different branches at $b$.
\end{enumerate}
\end{lemma}
\begin{proof}
Statement (1) follows from \cite[Lemma 2.1.2]{C},
(2) is proved in \cite[Lemma 1.2.2]{C} and
(3) comes from \cite[Lemma 2.2.2]{C}.
\end{proof}

\begin{lemma}
\label{!geo}
If $M$ is an $\R$-tree and $E_1$, $E_2$ are disjoint, closed, non-empty subtrees of $M$,
then there exists a unique shortest geodesic segment $[u, v]$ with $u\in E_1$ and $v\in E_2$.  
Moreover, for all $b\in E_1$ and $c\in E_2$, the geodesic segment from $b$ to $c$ must contain $[u, v]$.
\end{lemma}
\begin{proof}
This is \cite[Lemma 2.1.9]{C}.
\end{proof}

The preceding lemma directly implies the following fact, used often in this paper:
Given an $\R$-tree $M$, a closed subtree $E$ and a point $a\in M$,
there is a unique point $e$ in $E$ closest to $a$.  In that situation, one has that $\dist(a, E)=d(a, e)$
and also that $e$ is on the segment $[a,b]$ for every point $b \in E$.  Equivalently, given any $b \in E$,
if $\gamma \colon [0,d(a,b)] \to M$ is an isometry and has the geodesic segment $[a,b]$ as its image, 
then $e = \gamma(s)$, where $s = d(a,e)$ is the least value of $t \in [0,d(a,b)]$ for which $\gamma(t) \in E$.

\begin{definition}
\label{piecewise segment}
Let $x_0, x_1, x_2, \dots, x_n$ be points in an $\R$-tree $M$ and
$\gamma \colon [0, d(x_0, x_n)] \rightarrow M$  the isometric embedding
with $\gamma(0)=x_0$ and $\gamma(d(x_0, x_n))=x_n$ that has image
equal to the geodesic segment $[x_0, x_n]$. If for each $i=0,\dots,n$
we have $x_i=\gamma(a_i)$ where $0=a_0\leq a_1\leq \dots\leq a_n=d(x_0, x_n)$,
then we write
$[x_0, x_n]=[x_0, x_1, \dots, x_n]$, and call $[x_0, x_1, \dots, x_n]$ a \emph{piecewise segment}.

In other words, if $x_0, x_1,\dots, x_n$ are elements of $[x_0, x_n]$ listed
in increasing order of distance from $x_0$, then we  write $[x_0, x_n]=[x_0, x_1, \dots, x_n]$
for this \emph{piecewise segment.}
Note that we also know
$[x_1, x_n]=\bigcup_{i=1}^{n-1} [x_i, x_{i+1}]$, and  by Lemma 2.1.4 in Chiswell \cite{C} we have that
$[x_1, x_n]=[x_1,x_2,\dots,x_n]$ if and only if $d(x_1, x_n)=\sum_{i=1}^{n-1} d(x_i, x_{i+1})$.

\end{definition}

\begin{lemma}
\label{differentclosestpoints}
Let $E$ be a closed subtree of an $\R$-tree $M$ and let $a,b\in M$. Let $e_a\in E$ be the
unique closest point to $a$, and let $e_b\in E$ be the unique closest point to $b$. 
If $e_a\not=e_b$, then 
\[
d(a,b)=d(a, e_a)+d(e_a, e_b)+d(b, e_b) \text{.}
\]
That is, $[a, e_a, e_b, b]$ is a piecewise segment.
\end{lemma}
\begin{proof}
Follows from Lemmas \ref{branch} and \ref{!geo}.
\end{proof}

Recall that $(X,d)$ is a \emph{pseudometric space} if $d \colon X^2 \to \R^{\geq 0}$ is
symmetric, satisfies the triangle inequality, and has $d(x,x)=0$ for all $x \in X$.  Its quotient metric space is
obtained by identifying $x,y \in X$ iff $d(x,y)=0$.

\begin{definition}[{{\it Gromov product\/}}]
\label{gromovproduct}
For a pseudometric space $(X,d)$ and $x,y,w \in X$, define
$$(x \cdot y)_w=\frac{1}{2}[d(x,w)+d(y,w)-d(x,y)] \text{.} $$
\end{definition}

It follows easily from Lemmas \ref{branch} and \ref{!geo} that in an $\R$-tree, 
the Gromov product $(x\cdot y)_w$ computes the distance from $w$ to the geodesic segment $[x,y]$.

\begin{definition}
\label{defhyperbolic}
Let $\delta \geq 0$. A pseudometric space $(X,d)$ is {\it $\delta$-hyperbolic} if,
for all $ x,y,z,w\in X$
\[
\min\{(x\cdot z)_w, (y\cdot z)_w\}-\delta \leq (x\cdot y)_w \text{.}
\]
\end{definition}

Note that if $(X,d)$ is a pseudometric space, then the quotient metric space
of $(X,d)$ is $\delta$-hyperbolic if and only if $(X,d)$ is $\delta$-hyperbolic.

In this paper we are nearly always interested in $0$-hyperbolic spaces,
and it is sometimes useful to have the following alternate characterization:

\begin{lemma}
\label{four point condition}
A pseudometric space $(X,d)$ is $0$-hyperbolic if and only if it satisfies
the following \emph{$4$-point condition}: for all $x, y, z, t\in X$
\[
d(x, y)+d(z, t)\leq \max\{d(x, z)+d(y, t), d(y, z)+d(x, t)\} \text{.}
\]
\end{lemma}
\begin{proof}
See Chiswell \cite[Lemma 1.2.6]{C} with $\delta=0$.
\end{proof}

\begin{lemma}
\label{triangle property}
If $M$ is a geodesic metric space, then $M$ is $0$-hyperbolic
if and only if, given any $a,b,c \in M$ and geodesic segments
$[a,b], [b,c],$ and $[c,a]$, the segment $[a,b]$ is contained in
$[b,c]\cup[c,a]$. 
\end{lemma}
\begin{proof}
See the proof of Proposition III.H.1.22 in Bridson and Haefliger \cite{BH}.
\end{proof}

\begin{definition}
\label{tree spanned} 
If $A \subseteq M$ is a subset of the $\R$-tree $M$, let $E_A$ denote the smallest subtree containing $A$.
We call this the \emph{$\R$-tree spanned by $A$.}
Note that
\[
E_A=\bigcup\{[a_1,a_2]\mid a_1, a_2\in A\} \text{.}
\]
The closure $\overline{E_A}$ of $E_A$ is the smallest closed subtree containing $A$.
\end{definition}

Note that if $A$ is finite, then $E_A=\overline{E_A}$ and $E_A$ is complete and has
finite diameter.  

\begin{definition}
An $\R$-tree $M$ is {\it finitely spanned}
if there exists a finite subset $A\subseteq  M$ such that $M=E_A$.
\end{definition}

\begin{lemma}
\label{0hyp}
(1)  A metric space is an $\R$-tree if and only if it is $0$-hyperbolic and geodesic.
\newline
(2) Any $0$-hyperbolic metric space
embeds isometrically in an $\R$-tree.
\newline
(3) Let $(X,d)$ be a $0$-hyperbolic metric space.  For $i=1,2$, suppose $f_i \colon (X,d) \to (M_i,d_i)$ are 
isometric embeddings into $\R$-trees, and let $E_i$ be the smallest subtree of $M_i$ containing $f_i(X)$.  Then there
is a unique isometry $g$ from $E_1$ onto $E_2$ such that $g \circ f_1 = f_2$.  
\end{lemma}
\begin{proof}
(1)(2) These are from Roe \cite[Proposition 6.12]{R}.
\newline
(3)  Let $(M,d)$  be any $\R$-tree extending $(X,d)$ and let $E$ be the smallest subtree of $M$ containing $X$.

For any points $u,u' \in E$ there exist $x,y,x',y' \in X$ such that $u \in [x,y], u' \in [x',y']$. (See Definition \ref{tree spanned}.)
The main point is to show that $d(u,u')$ is determined by the distances $d(x,u),d(x',u')$ together with the
pairwise distances between points in $\{x,y,x',y'\}$.  

Fix $u,u',x,y,x',y'$ as above.
By Lemma \ref{!geo} there exist $Y \in [x,y]$ and $Y' \in [x',y']$
such that  for any $v\in [x,y], v' \in [x',y']$ we have 
\begin{align*}
\tag*{(D)}  d(v,v') = d(v,Y)+d(Y,Y')+d(Y',v') \text{.}
\end{align*}
 (In particular, $[Y,Y']$ is the shortest geodesic segment among those between a point in $[x,y]$ and a point in $[x',y']$.)
 
Applying (D) to $x,x'$ and to $y,y'$ yields
\begin{align*}
d(x,x') &= d(x,Y)+d(Y,Y')+d(Y',x')\\
\tag*{and} d(y,y') &= d(y,Y)+d(Y,Y')+d(Y',y') \text{,}
\end{align*}
from which follows
\[
2d(Y,Y') = d(x,x')+d(y,y')-d(x,y)-d(x',y') \text{,}
\]
which shows that $d(Y,Y')$ is determined by the values of $d$ on $x,y,x',y'$.

Similarly, applying (D) to $x,y'$ and to $x,x'$ yields that $d(x,Y)$ is determined by the values 
of $d$ on $x,y,x',y'$.  Likewise, applying (D) to $x',y$ and to $x',x$ yields the same conclusion
for $d(x',Y')$.

Finally, applying (D) to the original pair $u,u'$ yields
\begin{align*}
d(u,u') &= d(u,Y)+d(Y,Y')+d(Y',u')\\
&= |d(x,Y)-d(x,u)| + d(Y,Y') + |d(x',Y')-d(x',u')|
\end{align*}
which gives the desired conclusion.

To construct the needed isometry $g$, there is an obvious definition from a segment 
of the form $[f_1(x),f_1(y)]$ in $M_1$ into $M_2$, for each $x,y \in X$, by taking
$g$ to be the unique isometry from $[f_1(x),f_1(y)]$ in $M_1$ onto $[f_2(x),f_2(y)]$ in $M_2$ that 
takes $f_1(x)$ to $f_2(x)$ and $f_1(y)$ to $f_2(y)$.  What is proved above shows 
that the union of all such maps is a well defined isometry from $E_1$ onto $E_2$ that satisfies $g \circ f_1 = f_2$.
Every isometry from $E_1$ to $E_2$ that satisfies $g \circ f_1 = f_2$ must agree with this map $g$.
\end{proof}

Given $a, b, c$ in an $\R$-tree, there is a unique point $Y$ so that $[a,b]\cap[a,c]=[a,Y]$.
In Chiswell \cite{C} after Lemma 2.1.2, it is shown that this $Y$ is also the unique point so that
$[b,a]\cap[b,c]=[b, Y]$ and $[c,a]\cap[c,b]=[c, Y]$, and that in fact,
$\{Y\}=[a,b]\cap[b,c]\cap[a,c]$.

\begin{notation}
\label{Ypoint}
If $a,b,c$ are points in an $\R$-tree, we denote by $Y(a,b,c)$
the unique point such that $\{Y\}=[a,b]\cap[b,c]\cap[a,c]$.
This point will be denoted simply by $Y$ when $a, b, c$ are understood.
\end{notation}

That an $\R$-tree is $0$-hyperbolic tells us that for any 3 points $a, b, c$
the segment $[a,b]$ is contained in $[b,c]\cup [c,a]$, by \ref{triangle property}
and \ref{0hyp}. Thus, 
the subtree $E$ spanned by $a, b, c$ 
is comprised of the segments $[a, Y], [b,Y]$ and $[c, Y]$, which share only the point
$Y=Y(a,b,c)$.  Therefore we have two possibilities: either $Y$ is one of $a,b,c$ and
$E$ is the segment connecting the other two; or $Y \not\in \{a,b,c\}$ and $E$ is ``Y-shaped",
explaining the notation introduced in \ref{Ypoint} above. (See Chiswell \cite[2.1.2]{C}.)

\begin{definition}
\label{D: endpoint}
Let $M$ be an $\R$-tree.  If $c\in M$ is such that there do \emph{not}
exist $a,b\in M\setminus \{c\}$ with
$c\in [a,b]$, then $c$ is called an {\it endpoint} of $M$.
Equivalently, an endpoint is a point with degree at most one.
\end{definition}

\begin{lemma}
\label{mingenset}
If an $\R$-tree $M$ is finitely spanned and $C$ is the set of endpoints of $M$, then
\begin{enumerate}
    \item if $B$ spans $M$, then $C\subseteq B$;
    \item the set $C$ spans $M$.
\end{enumerate}
Thus, $C$ is finite, and it is the unique smallest set that spans $M$.
\end{lemma}
\begin{proof}
Let $M$ be a finitely spanned $\R$-tree. Let $D$ be the diameter of $M$.
Let $B$ be a set that spans $M$.\\
Proof  of (1): Assume there is an endpoint $c\in M$ not contained in $B$.
Then there must exist $a, b\in B$ such that $c\in [a,b]$ and $c \neq a,b$.  But, this is a contradiction
because $c$ is an endpoint.\\
Proof of (2): Let $a\in M$.
Let $S_a$ be the set of all segments $[b, c]\subseteq M$ such that $a\in [b,c]$
and order $S_a$ by inclusion.
This is a partial ordering on $S_a$.
Let $\{[b_i,c_i]\mid i\in \alpha\}$ be a chain in this partial ordering, where $\alpha$ is a cardinal.
Let $I$ be the closure of $\bigcup_{i\in \alpha}[b_i, c_i]$.  Then $I$ is a geodesic segment in $M$.
Clearly $a\in I$, and the length of $I$ is at most $D$. Therefore $I\in S_a$, and
$I$ is an upper bound for the chain. The chain was arbitrary, 
so any chain has an upper bound. Therefore, by Zorn's Lemma
there exists a maximal element of $S_a$.  Let $[b_a,c_a]$ denote such a maximal element.
The elements $b_a$ and $c_a$ must be endpoints of $M$.  Say, for instance,
that $b_a$ is not an endpoint.  Then there exist $e, f\in M$ such that
$b_a\in[e,f]$ and $b_a \neq e,f$, and then either $[e,c_a]$ or $[f,c_a]$ will properly contain $[b_a, c_a]$.
This would mean $[b_a, c_a]$ was not maximal in $S_a$.
Therefore, for each $a\in M$, there exist endpoints
$b_a$ and $c_a$ so that $a\in [b_a, c_a]$.
So, $M$ is spanned by the set of its endpoints, and this spanning
set is as small as possible by (1).
\end{proof}

As noted above, if $(X,d)$ is a $\R$-tree, then so is its completion $(\overline X,d)$.
The next result gives information about the structure of $\overline X$.

\begin{lemma}
\label{completing a tree}
Let $(X,d,p)$ be a pointed $\R$-tree and $(\overline X,d,p)$ its completion, and suppose
$x \in \overline X \setminus X$.  Then:
\newline
(1) there exist $(x_n)$ in $X$ converging to $x$ such that $[p,x_1,\dots,x_n,x]$ is a
piecewise segment in $\overline X$ for all $n \geq 1$; and
\newline
(2) $x$ is an endpoint in $\overline X$.
\end{lemma}
\begin{proof}
(1) Given $x \in \overline X \setminus X$, let $(y_n)$ be a Cauchy sequence in $X$ that
converges to $x$ in the $\R$-tree $\overline X$.  Without loss of generality we may assume that
$(d(y_n,x))$ is decreasing toward $0$.  For each $n$, let $x_n$ be the
closest point in $\overline X$ on the segment $[p,x]$ to $y_n$.  By Lemma \ref{!geo}, $x_n$ is on
the segment $[p,y_n]$ in $\overline X$ and thus $x_n \in X$.  Further, from the
same Lemma and the assumptions on $(y_n)$ we conclude that $[p,x_1,\dots,x_n,x]$
is a piecewise segment in $\overline X$ for each $n$ and $(x_n)$ converges to $x$.

(2) If $x$ were not an endpoint in $\overline X$, there would be $y \in \overline X$ such
that $p$ and $y$ were in different branches in $\overline X$ at $x$.  By part (1), there must
exist $z \in X$ on the segment $[p,y]$ in $\overline X$ such that $z$ is closer to $y$ than $x$.
That is, $[p,x,z,y]$ would be a piecewise segment in $\overline X$.  But then $x$ would be 
on the segment $[p,z]$, which lies entirely in $X$.  This contradiction completes the proof.
\end{proof}
 
\section{Some continuous model theory}

We investigate the model theory of $\R$-trees using
continuous logic for metric structures as presented in 
Ben Yaacov, Berenstein, Henson and Usvyatsov \cite{N} and 
Ben Yaacov and Usvyatsov \cite{BU}. 
In this section we remind the reader of a few key
concepts and results from those papers, and then we present 
a few facts about existentially closed models and model
companions that are not discussed there.
For the rest of this section we fix a continuous first
order language $L$.

As explained in \cite[Section 3]{N}, in continuous model theory it is required that 
\emph{structures} and \emph{models} are metrically complete.  However, formulas 
and conditions are evaluated more generally in \emph{pre-structures}, as explained 
in \cite[Definition 3.3]{N}. Further, it is shown in \cite[Theorem 3.7]{N} that the 
completion of a prestructure is an elementary extension.  In this paper we use 
notation of the form $\M \models \theta$ only when $\M$ is a structure; 
in other words, $M$ must be metrically complete.

  Next, some reminders about saturation in continuous logic.
  A set $\Sigma(x_1,\dots,x_n)$ of $L$-conditions (with free variables
  among $x_1,\dots,x_n$) is called {\it satisfiable in $\M$} if
  there exist $a_1,\dots,a_n$ in $M$ such that $\M\models E[a_1,\dots,a_n]$
  for every $E(x_1,\dots,x_n)\in \Sigma$.
  Let $\kappa$
  be a cardinal. A model $\M$ of $T$ is called {\it $\kappa$-saturated} if
  for any set of parameters $A\subseteq M$ with cardinality $<\kappa$
  and any set $\Sigma(x_1,\dots,x_n)$ of $L(A)$-conditions, if every finite subset
  of $\Sigma(x_1,\dots,x_n)$ is satisfiable in $(\M,a)_{a\in A}$, then the entire set $\Sigma(x_1,\dots,x_n)$
  is satisfiable in $(\M,a)_{a\in A}$.
 
   \begin{proposition}
  For any countably incomplete ultrafilter $U$ on $I$, the $U$-ultraproduct
  of a family of $L$-structures $(\M_i\mid i\in I)$ is $\omega_1$-saturated.
  \end{proposition}
\begin{proof}
See Ben Yaacov, Berenstein, Henson and Usvyatsov \cite[Proposition 7.6]{N}.
\end{proof}
  Note that any non-principal ultrafilter on $\N$ is countably incomplete.
  
  \begin{proposition}
  For any cardinal $\kappa$, any $L$-structure $\M$ has a $\kappa$-saturated elementary extension.
  \end{proposition}
\begin{proof}
See \cite[Proposition 7.10]{N}.
\end{proof}

\begin{definition}
An \emph{$\inf$-formula} of $L$ is a formula of the form
$$\inf_{y_1}\dots\inf_{y_n}\phi(x_1,\dots,x_k,y_1,\dots,y_n)$$
where $\phi(x_1,\dots,x_k, y_1,\dots,y_n)$ is quantifier-free.
\end{definition}
A $\sup$-formula of $L$ is defined similarly. These $\sup$-formulas
are the universal formulas in continuous logic. For an $L$-theory $T$, we use the
notation $T_\forall$ for the set of universal sentences $\sigma$ ($\sup$-formulas with no free variables) 
such that the condition $\sigma=0$ is implied by the theory $T$.  
Note that, as in classical first order logic, $T_\forall$
is the theory of the class of $L$-substructures of models of $T$.

\begin{definition}
Let $T$ be an $L$-theory and suppose $\M \models T$. We say $\M$ is
an {\it existentially closed (e.c.)} model of $T$ if, for any
$\inf$-formula $\psi(x_1,\dots,x_m)$,
any $a_1,\dots,a_k \in M$, and any $\n \models T$ that is an extension of $\M$,
we have $\psi^\n(a_1,\dots,a_m) = \psi^\M(a_1,\dots,a_m)$.
\end{definition}

An $L$-theory $T$ is {\it model complete} if any embedding between models
of $T$ is an elementary embedding.

\begin{proposition}
\label{allec}
The $L$-theory $T$ is model complete if and only if
every model of $T$ is an existentially closed model of $T$.
\end{proposition}
\begin{proof}
This is Robinson's Criterion for model completeness.  The proof given by
Hodges \cite[Theorem 8.3.1]{Ho} 
for classical first order logic can easily be adapted to the continuous setting.
\end{proof}

In Ben Yaacov \cite[Appendix A]{IN} there is some further discussion of 
$\inf$- and $\sup$-formulas and of model completeness.

\begin{definition}
\label{defnmodelcompanion} Let $T$ be an $L$-theory. A {\it model
companion} of $T$ is an $L$-theory $S$ such that:
\begin{itemize}
\item every model of $S$ embeds in a model of $T$;
\item every model of $T$ embeds in a model of $S$;
\item $S$ is model complete.
\end{itemize}
\end{definition}

Note that the first two criteria in this definition together
are equivalent to the statement $S_{\forall}=T_{\forall}$.
As in classical first order logic, if a theory has a model
companion, then that model companion is unique (up to equivalence of
theories).

Recall that a theory $T$ is \emph{inductive} if whenever $\Lambda$ is a linearly ordered set
and $(\M_\lambda\mid\lambda \in \Lambda)$ is a chain of
models of $T$, then the completion of the union of $(\M_\lambda\mid\lambda \in \Lambda)$ is a model of $T$.

\begin{proposition}
\label{axiomatizeec}
Let $T$ be an inductive $L$-theory
and let $\mathcal{K}$ be the class of existentially closed models of $T$.
If there exists an $L$-theory $S$ so that $\mathcal{K}=\Mod(S)$, then
$S$ is the model companion of $T$.
\end{proposition}
\begin{proof}
The proof of \cite[Theorem 8.3.6]{Ho} can be adapted to the continuous setting.
\end{proof}

We say the $L$-theory $T$ has \emph{amalgamation over substructures} if
for any substructures  $\M_0$, $\M_1$ and $\M_2$ of models of $T$ and
embeddings $f_1\colon \M_0\rightarrow \M_1$, $f_2\colon \M_0\rightarrow \M_2$,
there exists a model $\n$ of $T$  and embeddings
$g_i\colon \M_i\rightarrow \n$ such that $g_1\circ f_1=g_2\circ f_2$.

\begin{proposition}
\label{modelcompanionplusamalg}
Let $T_1$ and $T_2$ be $L$-theories such that $T_2$
is the model companion of $T_1$.  Assume $T_1$ has amalgamation over substructures.
Then $T_2$ has quantifier elimination.
\end{proposition}
\begin{proof}
The corresponding result in classical first order logic is the
equivalence of (a) and (d) in \cite[Theorem 8.4.1]{Ho}.  The proof
given there can be adapted to the continuous setting.
\end{proof}






\section{The theory of pointed \texorpdfstring{$\R$}{R}-trees with radius at most \texorpdfstring{$r$}{r}}
\label{theory of R-trees}

By the \emph{radius} of a pointed metric space $(M,p)$ we mean the
supremum of the distance $d(p,x)$ as $x$ varies over $M$, and when using the term
we always require that this expression is finite.

In this section we first present the continuous
signature used in this paper to study $\R$-trees.  We then give
axioms for the theory $\RT_r$ of $\R$-trees with radius $\leq r$.

Let $r>0$ be a real number. Define the signature $L_r:=\{p\}$ where $p$ is a 
constant symbol and specify
that the metric symbol $d$ has values which lie in the interval $[0,2r]$.
Any pointed metric space $(M, p)$ with radius $\leq r$ naturally
gives rise to an $L_r$-prestructure $\M=(M,d,p)$, in which $d$ is a metric;
$\M$ is an $L_r$-structure if the metric space involved is
metrically complete.

\begin{remark}
\label{endpoints at distance r}
For future reference, we note: if $M$ is a pointed $\R$-tree of radius at most $r$ and if $x \in M$ satisfies
$d(p,x) = r$, then $x$ is an endpoint in $M$.  (See Definition \ref{D: endpoint}.)  Indeed, if $x$ is not an 
endpoint then we can let $y \in M$ be an element of a branch at $x$ that does not contain $p$.  But
then we would have $d(p,y) = d(p,x)+d(x,y) > r$.
\end{remark}

Next we define a set of axioms $\RT_r$ for the class of complete,
pointed $\R$-trees of radius $\leq r$. 
Recall the connective
$\dotminus\colon [0,\infty)\times[0,\infty) \rightarrow [0,\infty)$
defined by $x\dotminus y=\max\{x-y,0\}$.

\begin{definition}
\label{rtreetheory}
Let $\RT_r$ be the $L_r$-theory consisting of the following conditions:
\begin{enumerate}

    \item \label{boundax}
    \hfil$\sup_{x} d(x,p) \leq r$;\vspace{0.25cm}

    \item \label{midptax} 
    \hfil{$\sup_{x}\sup_{y}\inf_{z}\max\{
    \big|d(x,z)-\frac12 d(x,y) \big|,\ \ \big| d(y,z)-\frac12 d(x,y) \big|\}=0$;}\vspace{0.25cm}

    \item \label{hypax} 
    \hfil{$\sup_{x}\sup_{y}\sup_{z}\sup_{w}
    \big( \min\{(x\cdot z)_w, (y\cdot z)_w\}\dotminus (x\cdot y)_w\big)=0$.}

\end{enumerate}
\end{definition}

Condition (2) formalizes the approximate midpoint property.
In reading (3), recall that $(x\cdot y)_w$ denotes the Gromov product (see Definition \ref{gromovproduct}),
which is given by an explicit formula $\phi(x,y,w)$ in the signature $L_r$.

\smallskip
The next lemma shows that the class of complete pointed $\R$-trees of
radius $\leq r$ is axiomatized by $\RT_r$.

\begin{lemma}
\label{modelsofT}
The models of $\RT_r$ are exactly the complete, pointed $\R$-trees of radius $\leq r$. 
\end{lemma}
\begin{proof}
First we assume $\M\models\RT_r$.  Then $(M, d,p)$ is a complete, pointed
metric space. Axiom (\ref{boundax}) guarantees that $M$ has radius $\leq r$. 
Axiom \ref{hypax} implies that $M$ is $0$-hyperbolic.

Axiom (\ref{midptax}) implies that for any $x,y\in M$ and
any $\epsilon >0$ there is $z\in M$ such that $d(x,z)$ and $d(y,z)$ are within $\epsilon$
of $d(x,y)/2$. We show that in a complete metric space that is $0$-hyperbolic,
this implies the exact midpoint property.  (See Fact \ref{midptfact}.)
Given $x,y \in M$, for each $n$ let $z_n \in M$
be such that $d(x,z)$ and $d(y,z)$ are within $1/n$
of $d(x,y)/2$.   Applying the $4$-point condition (see Lemma \ref{four point condition})
we have
\[
d(z_m,z_n) + d(x,y) \leq \max \big[ d(z_m,x) + d(z_n,y), d(z_m,y)+d(z_n,x) \big]
\]
\[
\leq \max \big[ d(x,y) + \frac1m +\frac1n, d(x,y) +\frac1m +\frac1n] = d(x,y) +\frac1m +\frac1n
\]
from which we get $d(z_m,z_n) \leq 1/m + 1/n $ for all $m,n$.  Therefore $(z_n)$
converges in $M$ to an exact midpoint between $x$ and $y$.

Therefore, by Lemma \ref{0hyp}(1), $\M$ is a pointed $\R$-tree with radius $\leq r$.

That a complete, pointed $\R$-tree with radius $\leq r$ is a model of $\RT_r$ is clear.
\end{proof}

\begin{remark}
Structures in continuous logic are required to be metrically complete, while
in general, $\R$-trees are not complete.
A pointed $\R$-tree $M$ with radius $\leq r$ can naturally be viewed as an
$L_r$-prestructure, which is an $L_r$-structure iff it is complete (since the
pseudometric on the prestructure is actually a metric). If $M$ is not complete,
then its metric completion is known to be an $\R$-tree (see Chiswell \cite[Lemma 2.4.14]{C}).
Further, the completion of a prestructure is known to be an elementary extension,
and therefore the prestructure and its completion are completely equivalent from a 
model-theoretic perspective. (See 
Ben Yaacov, Berenstein, Henson and Usvyatsov \cite[pages 15--17]{N}.)
Note that this also means any two pointed $\R$-trees of radius $\leq r$ that have the
same metric completion are indistinguishable from a model-theoretic
perspective, and that a metrically complete $\R$-tree can be identified model-theoretically 
with any of its dense sub-prestructures.  (However, those metric sub-prestructures are not
necessarily $\R$-trees. In particular, they are not necessarily geodesic spaces.)
\end{remark}

We close this section by noting a property of $\RT_r$ that will be used later.

\begin{lemma}
\label{chainofmodels}
The theory $\RT_r$ is inductive. That is, the completion of the union
of an arbitrary chain of models of $\RT_r$ is a model of $\RT_r$.
\end{lemma}
\begin{proof}
The proof of Chiswell \cite[Lemma 2.1.14]{C} can be modified to show that the
union of an arbitrary chain of pointed $\R$-trees is again a pointed $\R$-tree.
Also, the completion of an $\R$-tree is an $\R$-tree.  (See \cite[Lemma 2.4.14]{C}.)
Since base points are preserved by embeddings
of models, the radius of the underlying pointed $\R$-tree for the union of a chain is most $r$. 

Alternatively, note that $\RT_r$ is an $\forall \exists$-theory, and therefore the
class of its models is closed under completions of unions of chains.
\end{proof}

\section{Some definability in \texorpdfstring{$\RT_r$}{RTr}}
\label{definable}

We now discuss the notion of \emph{definability} for subsets of and functions
on the underlying $\R$-tree $M$ of a model $\M$ of $\RT_r$.  For background
on definable predicates, sets and functions see 
Ben Yaacov, Berenstein, Henson and Usvyatsov \cite{N}.
The first result shows that every closed ball centered at the base point
is uniformly quantifier-free $0$-definable in models of $\RT_r$.

\begin{lemma}
\label{geodballs}
Let $s\in [0,r]$. Let $\phi(x)$ be the quantifier-free formula $d(x,p)\dotminus s$.  Suppose $\M\models \RT_r$ and 
$(M,d,p)$ is the underlying $\R$-tree of $\M$.  Evidently the zeroset of $\phi(x)^\M$ is equal to the closed
ball $B_s(p)$ of radius $s$ centered at $p$ in $M$.  
Then for all $x\in M$ we have $\dist(x, B_s(p))=\phi(x)^\M$.

Therefore, the closed ball $B_s(p)\subseteq M$ is uniformly $0$-definable with respect to the theory $\RT_r$.
\end{lemma}
\begin{proof}
It suffices to show that $\phi^\M(x)$ is equal to the distance function
$\dist(x,B_s(p))$.
For $x\in M$ we know $\phi^\M(x)=0$ if and only if $d(x,p)\leq s$,
{ie}, if and only if $x\in B_s(p)$.
Now, let $x\notin B_s(p)$. Then $\phi^\M(x)=d(x,p)-s$.
Let $\gamma$ be a geodesic segment from $p$ to $x$ with $\gamma(0)=p$.
Then
\[
\dist(x,B_s(p))\leq d(x,\gamma(s))=d(x,p)-d(p, \gamma(s))=d(x,p)-s=\phi^\M(x) \text{.}
\]
Now toward a contradiction, assume $d(x,p)-s >\dist(x, B_s(p)).$
Then there exists a point in $c\in B_s(p)$ with $d(c,x)<d(x,p)-s$.  This implies,
\[
d(x,p)\leq d(p,c)+d(c,x)\leq s+d(c,x)<s+d(x,p)-s=d(x,p)
\]
so $d(x,p)<d(x,p)$, which is a contradiction. Therefore, $\dist(x, B_s(p))=\phi(x)^\M.$
\end{proof}

Note that the preceding argument only needs that the underlying metric space 
$M$ is a geodesic space.  

\smallskip
Next, we present a short discussion of some specific definable functions,
points and sets in models of $\RT_r$.

Let $M$ be an $\R$-tree and for $s\in [0,1]$ define the function
$\nu_s\colon M\times M\rightarrow M$ by:
$\nu_s(x_1,x_2)=$ the point in $[x_1,x_2]$ with distance $sd(x_1, x_2)$ from $x_1$
and distance $(1-s)d(x_1, x_2)$ from $x_2$. 

\begin{lemma}
\label{definablemidpoints}
Let $s\in [0,1]$. The function $\nu_s\colon M\times M \rightarrow M$
is uniformly $0$-definable in models $\M$ of $\RT_r$ via a quantifier-free formula.
\end{lemma}
\begin{proof}
Let $\psi$ be the formula
$$\max\{d(x_1,y)\dotminus sd(x_1,x_2), d(x_2,y)\dotminus (1-s)d(x_1,x_2)\} \text{.} $$
Let $\M\models \RT_r.$
In $\M$, the distance $d(\nu_s(x_1,x_2),y)$ is equal to $\psi^\M(x_1,x_2,y).$
So the function $\nu_s$ is $0$-definable via this quantifier-free formula in any model of $\RT_r$.
\end{proof}

In Definition \ref{Ypoint} we defined $Y=Y(a,b,c)$, the unique point so that
$\{Y\}=[a,b]\cap[b,c]\cap[a,c]$. Next we show that $Y$ is a definable function.

\begin{theorem}
The function $Y\colon M^3\rightarrow M$ 
that for inputs $x_1, x_2, x_3\in M$ returns $Y(x_1, x_2, x_3)$ is uniformly
$0$-definable in models $\M$ of $\RT_r$ via a quantifier-free formula.
\end{theorem}
\begin{proof}
Recall that in an $\R$-tree, the Gromov
product $(a \cdot b)_x$ is equal to the distance from $x$ to the segment $[a,b]$,
and this distance is realized by a unique closest point on $[a,b]$.  Let
\[
\phi(x_1, x_2, x_3, x)=\max\left\{(x_1\cdot x_2)_x, (x_1\cdot x_3)_x, (x_2\cdot  x_3)_x\right\} \text{.}
\]
Let $\M\models \RT_r$ and $x_1, x_2, x_3\in M$. We will show that in $\M$, 
\[
d( x, Y(x_1, x_2, x_3))=\phi^\M(x_1, x_2, x_3, x) \text{.}
\]
Let $E$ be the closed subtree of $M$ spanned by $x_1, x_2$ and $x_3$. Abbreviate 
$Y(x_1, x_2, x_3)$ by $Y$.
For $x\in M$, let $z$ be the unique point closest to $x$ in $E$. Then we have $d(x, Y)=d(x, z)+d(z, Y)$. 
The point $z$ must lie on at least one of $[x_1, Y]$, $[x_2, Y]$ or $[x_3, Y]$.
Without loss of generality, assume $z\in [x_1, Y]=[x_1, x_3]\cap[x_1, x_2]$. 

Since $z$ is closest in $E$ to $x$  and $z\in [x_1, x_3]\cap[x_1, x_2]$ we know 
$d(x, z)=(x_1\cdot x_2)_x=(x_1\cdot x_3)_x$.
That $z$ is closest in $E$ to $x$ also implies $(x_2\cdot x_3)_x\geq d(x,z)$. 
This makes $(x_2\cdot x_3)_x\geq(x_1\cdot x_2)_x=(x_1\cdot x_3)_x$.
Thus, $\phi^{\M}(x_1, x_2, x_3, x)=(x_2\cdot x_3)_x$.

Since $z\in [x_1, Y]$, the point $Y$ is the closest point to $z$ on $[x_2, x_3]$, and 
therefore $Y$ is the closest point to $x$ on $[x_2, x_3].$ It follows that, $d(x, Y)=(x_2\cdot x_3)_x$.
We conclude that $\phi^{\M}(x_1, x_2, x_3, x)=d(x,Y).$
\end{proof}

\section{Amalgamation} 

Next, we discuss amalgamation for the $L_r$-theory $\RT_r$.
The following result from Chiswell \cite{C} discussing amalgamating over points in $\R$-trees 
is then extended to amalgamation of $\R$-trees over subtrees. This leads to a 
proof of amalgamation over substructures for $\RT_r$.

\begin{lemma}
\label{chiswell add to points}
Let $(X,d)$ be an $\R$-tree and let $\{(X_i, d_i)\mid i\in I\}$ be a family of $\R$-trees
such that
$X_i\cap X=\{x_i\}$ for all $i\in I$.
$X_i\cap X_j= \{x_i\}$ if $x_i=x_j$ and $X_i\cap X_j=\emptyset$ otherwise.
Define $N=(\bigcup_{i\in I} X_i)\cup X$ and define $\hat d\colon N\times N \rightarrow \R$ by 
\begin{itemize}
\item On $X\times X$, $\hat d=d$, and on $X_i\times X_i$, $\hat d=d_i$.
\item If $x\in X_i$ and $x'\in X_j$, with $i\not= j$, then
\[
\hat d(x,x')=\hat d(x',x)=d_i(x,x_i)+d(x_i, x_j)+d_j(x_j, x') \text{.}
\]
\item If $x\in X_i$ and $x'\in X$, then $\hat d(x,x')=\hat d(x',x)=d_i(x,x_i)+d(x_i, x')$.
\end{itemize}
Then $(N,\hat d)$ is an $\R$-tree.
\end{lemma}
\begin{proof}
See \cite[Lemma 2.1.13]{C}
\end{proof}

We next apply this lemma to prove $\R$-trees can be amalgamated over subtrees.

\begin{lemma}
\label{rtreeamalg}
Given $\R$-trees $M_0$, $M_1$ and $M_2$ and isometric embeddings
$f_i\colon M_0\rightarrow M_i$ for $i=1, 2$,  there exists an $\R$-tree $N$
and isometric embeddings $g_i\colon M_i\rightarrow N$ such that $g_1\circ f_1=g_2\circ f_2$.
\end{lemma}
\begin{proof}
Without loss of generality, we may assume that $M_0$ is nonempty, that $M_0\subseteq M_1$ and $M_0\subseteq M_2$,
that $M_1\cap M_2=M_0$, and that the $f_i$ are inclusion maps.  Further, we may assume that $M_i$ is
complete for $i=0,1,2$, since the maps $f_1,f_2$ are isometric, and can thus be extended as needed.
(Recall that the completion of an $\R$-tree is again an $\R$-tree.)
Let $\{x_k\mid k\in |M_0|\}$ be a $1-1$ list of the elements of $M_0$.
For each $k\in |M_0|$ let $B_k$ be the set of branches $\beta$ at $x_k$ in $M_2$ such that $\beta \cap M_0 = \emptyset$.
Let $\alpha_k$ be the cardinality of $B_k$ (which we may assume is nonempty, by extending $M_2$ if needed),
and let $A_k=\{\beta_{kv}\mid v\in \alpha_k\}$ be a $1-1$ list of the distinct  elements of $B_k$.
Then let $I=\{(k, v)\mid k\in|M_0|, v\in \alpha_k\}$,
and define $X_i=\beta_{kv}\cup \{x_k\}$ for each $i=(k,v)\in I$.
For each $i = (k,v) \in I$, let $x_i = x_k$.  
For each $i\in I$ we consider $X_i$ as a subtree of $M_2$.  

We claim that $X = M_1$ and the family $(X_i \mid i \in I)$ satisfy the hypotheses of Lemma \ref{chiswell add to points}.
By construction, every point in $M_0$ is contained in some $X_i$.
For each $i=(k,v)\in I$ we see that $X_i \cap M_1 = \{x_k\} = \{x_i\}$ since $x_k$ is the only point in $X_i \cap M_0$ and
$X_i \cap M_1 \subseteq M_2 \cap M_1 = M_0$.  Further, given any point $y \in M_2 \setminus M_0$, we may let $x$ be the closest point
to $y$ in $M_0$ and take $k$ so that $x_k=x$; then we may take $\beta$ to be a branch
in $M_2$ at $x_k$ that contains $y$.  It follows that $\beta \cap M_0 = \emptyset$; indeed, if $x' \in \beta \cap M_0$, then
$Y(x,x',y)$ would be an element of $M_0$ that was closer to $y$ than $x$.  Thus there exists $i = (k,v) \in I$ for
which $\beta = \beta_i$ and hence $y \in X_i$.  Thus we have $M_2 = \bigcup_{i \in I} \, X_i$.

Applying Lemma \ref{chiswell add to points}, we get that 
$N=(\bigcup_{i\in I} X_i)\cup M_1$ with the metric $\hat d$ is an $\R$-tree, and 
clearly $N=M_2\cup M_1$ as metric spaces. Define $g_i$ to be the inclusion of $M_i$ in $N$. 
Then $g_1\circ f_1=g_2\circ f_2$ is clear.
\end{proof}

To apply these results to models of $\RT_r$, we begin by
proving any subset of a model of $\RT_r$ gives rise to a unique model of $\RT_r$.
In particular, any substructure extends to a unique model.

\begin{lemma}
\label{substructureofrtree}
Let $\n\models \RT_r$. Any subset $A\subseteq N$ extends to a model $\M\models \RT_r$
such that whenever $f$ is an embedding of $A\cup\{p\}$ into $\W\models \RT_r$
that satisfies $f(p)=p$, there is a unique extension of $f$ that embeds $\M$
into $\W$.
\end{lemma}
\begin{proof}
Let $(N,d,p)$ be the underlying $\R$-tree of the model $\n$, and let $A\subseteq N$.   
Define the structure $\M=(M, d,p)$ by $M=\overline{E_{A\cup \{p\}}}$ with the metric and base point from $(N,d,p)$.
Then $M$ is the smallest closed subtree of $N$ containing $A\cup\{p\}$.  Obviously
$\M\models \RT_r$; Lemma \ref{0hyp}(3) yields a unique extension of $f$ to embed
$E_{A\cup \{p\}}$ into $\W$.  Further, since $f$ is an isometry, it extends further (also
uniquely) to embed $\M$ into $\W$. 
\end{proof}

\begin{theorem}
\label{amalg}
The $L_r$-theory $\RT_r$ has amalgamation over substructures.
That is, if  $\M_0$, $\M_1$ and $\M_2$
are substructures of models of $\RT_r$ and
$f_1\colon \M_0\rightarrow \M_1$, $f_2\colon \M_0\rightarrow \M_2$ are embeddings,
then there exists a model $\n$ of $\RT_r$  and embeddings
$g_i\colon \M_i\rightarrow \n$ such that $g_1\circ f_1=g_2\circ f_2$.
\end{theorem}
\begin{proof}
First, by Lemma \ref{substructureofrtree} we may assume that $\M_0$, $\M_1$ and $\M_2$
are models of $\RT_r$ with underlying complete, pointed $\R$-trees $(M_i, d_i, p_i)$.
Further, we may assume that $M_0\subseteq M_i$ and $p_0=p_i$ for $i=1,2$
and that the $f_i$ are inclusion maps.

We use Lemma \ref{rtreeamalg} to construct the $\R$-tree $(N, d)$ 
and isometric embeddings $g_i\colon M_i\rightarrow N$ such that $g_1\circ f_1=g_2\circ f_2$.
Recall that in this construction, $N=M_1\cup M_2$.
Define the base point of $N$ to be $q=p_1=p_2=p_0$.
Since every point in each $M_i$ has distance $\leq r$ from $p_i=q$, we conclude that 
the pointed $\R$-tree $(N,d,q)$ has radius $\leq r$.
Let $\n=(N,d,q)$ be the corresponding $L_r$-structure. 
Then $\n$ is a model of $\RT_r$, and since $g_i$ are isometric embeddings
preserving the base point, they give rise to embeddings of $L_r$-structures as required.
\end{proof}





\section{The model companion of \texorpdfstring{$\RT_r$}{RTr}}
\label{model companion}

In this section we define what it means for a pointed $\R$-tree of radius $\leq r$
to be {\it richly branching.}
We then show that the theory of richly branching pointed $\R$-trees
with radius $r$ is the model companion of $\RT_r$.
Throughout the rest of this paper we assume $r>0$.

\begin{definition}
\label{richlybranching}

A pointed $\R$-tree $(M, d, p)$ of radius $\leq r$ is {\it richly branching} if the set
$$B=\{b\in M\mid \text{ at } b \text{ there are at least 3 branches of height } \geq r-d(p, b)\}$$
is dense in $M$.
\end{definition}

\begin{remark}
\label{general richly branching}
The preceding definition and arguments below apply specifically to pointed $\R$-trees with 
radius $\leq r$.  For general $\R$-trees we define:
\emph{an $\R$-tree $M$ is {\it richly branching} if the set of points at which there are at least 3 branches of 
infinite height is dense in $M$.} Note that an $\R$-tree that is richly branching in this sense must be unbounded.

\begin{lemma}
\label{branch points on a segment}
Let $(X,d)$ be an $\R$-tree in which the set of branch points is dense.  Then for any distinct $x,y \in X$,
the set of branch points on $[x,y]$ is dense in $[x,y]$.
\end{lemma}
\begin{proof}
Let $z$ be on the segment $[x,y]$ and suppose $0 < \delta < \min[d(x,z),d(y,z)]$.
By assumption there is a branch point $w$ in $X$ with $d(w,z)<\delta$.  If $w$ lies
on $[x,y]$ we are done.  Otherwise, let $u$ be the point on $[x,y]$ that is closest to
$w$.  By Lemma \ref{!geo} we have that $x,y,w$ are in distinct branches at $u$,
and $d(z,u) \leq d(w,u)+d(z,u) = d(w,z) < \delta$.  That is, $u$ is a branch point on
$[x,y]$ that is arbitrarily close to an arbitrary point on $[x,y]$ that is distinct from $x,y$.
\end{proof}

If we start with an unbounded richly branching $\R$-tree $M$, assign an arbitrary base point $p$ and select $r>0$, 
we can make a richly branching $\R$-tree with radius $\leq r$ as in Definition \ref{richlybranching} by taking 
the closed ball of radius $r$ with center $p$ in $M$. Note that this yields an $L_r$-prestructure
and the completion is an $L_r$-structure.

Conversely, suppose $M$ is an $\R$-tree and $p$ is any point in $M$.  If the closed ball of radius $r$ with 
center $p$ in $M$ is richly branching in the sense of \ref{richlybranching} for an unbounded set of $r>0$, 
then $M$ is richly branching as a general $\R$-tree.
(See Lemma \ref{branches of sufficient height} below.) 
\end{remark}

Next we give axioms for the class of complete, richly branching, pointed $\R$-trees with radius $r$.

\begin{definition}
\label{rbrtaxioms}
Define ${\psi}(x)$ to be the $L_r$-formula
\[
\inf_{y_1y_2y_3}\max\left\{\max_{i=1,2,3}\{|d(x,y_i)-(r-d(p,x))|\},
\max_{1\leq i<j \leq 3}\{ d(x,y_i)+d(x,y_j)-d(y_i,y_j) \}\right\}
\]
and let $\phi=\displaystyle \sup_{x} {\psi}(x)$.
\end{definition}

\begin{definition}
\label{rbrttheory}
Let $\rbRT_r=\RT_r\cup \{\phi=0\}$.
\end{definition}

To help parse these axioms and picture what they mean, note that in an $\R$-tree
\[
d(x,y_i)+d(x,y_j)-d(y_i,y_j)=2(y_i\cdot y_j)_x = 2\dist(x, [y_i, y_j]) \text{.}
\]
Consider the case when the infima are realized exactly
in $\M$ (eg, if $\M$ is an $\omega$-saturated model.) In this setting, $\phi=0$ means that for any element $a\in M$
with $d(p, a)<r$, there exist $b_1, b_2, b_3$ each on a separate branch at $a$ and each
with distance $r-d(p,a)$ from $a$.
We know $b_1, b_2, b_3$ are on distinct branches at $a$  because $d(a,b_i)>0$  for every $1 \leq i \leq 3$ and 
$2\dist(a, [b_i, b_j])=0$ for every $1\leq i<j\leq 3$. In particular, this makes $a$ a branch point.
In a general model where the infima are not necessarily realized, an approximate version
of this is true; given an element $a$ there must be a branch point within $\epsilon$ of $a$, as we now show.

\begin{lemma}
\label{branching point within epsilon}
Let $\M\models \rbRT_r$ with underlying pointed $\R$-tree $(M, d, p)$ and let $a\in M$. Let $h\in \R$ 
be such that $0<h<r-d(p,a)$. For any $\epsilon>0$ there exists a point
$b\in M$ so that $d(a,b)<\epsilon$ and there are 3 branches at $b$,
each with height at least $h$. 
\end{lemma}
\begin{proof}
Let $s=r-d(p, a)$, and assume $0<4\epsilon< s-h$.
Since $\phi^{\M}=0$, we know $\psi(a)^{\M}=0$.
Thus there exist points $c_1, c_2, c_3\in M$ so that 
$$|d(a,c_i)-s|<\epsilon \text{ and } d(a,c_i)+d(a,c_j)-d(c_i,c_j)<\epsilon \text{.}$$
It follows that for $i\not=j$ 
$$d(c_i, c_j)>d(a, c_i)+d(a, c_j)-\epsilon >2s-3\epsilon>0$$ 
meaning 
$c_1, c_2, c_3$ must be distinct. Moreover, $c_1, c_2$
and $c_3$ cannot lie along a single piecewise segment.
To see this assume $[c_i, c_k]=[c_i, c_j, c_k]$  for $i,j,k \in \{1, 2, 3\}$. Then 
$d(c_i, c_k)=d(c_i, c_j)+d(c_j, c_k)$, implying
$d(c_i, c_k)=d(c_i, c_j)+d(c_j, c_k)>4s-6\epsilon$.
However, $d(c_i, c_k)\leq d(a, c_i)+d(a, c_k)<2s+2\epsilon$.
These inequalities imply $4\epsilon >s>s-h$, a contradiction.  So, each of
$c_1, c_2, c_3$ lies on
a different branch at $Y(c_1, c_2, c_3)$.
Thus, there are at least 3 branches at $Y(c_1, c_2, c_3)$.  

Next, let $u=Y(a, c_1, c_2)$,  $v=Y(a, c_2, c_3)$ and $w=Y(a, c_1, c_3)$.
Without loss of generality, we assume $d(a, u)\leq d(a, v)\leq d(a, w)$; in other
words, $\dist(a, [c_1, c_2])\leq \dist(a, [c_2, c_3])\leq \dist (a, [c_1, c_3])$.
Then $\dist(a, [c_1, c_2])=\dist(a, [c_2, c_3]$), since $M$ is $0$-hyperbolic.
Lemma 2.1.6 in Chiswell \cite{C} and the subsequent discussion
yield that $u=v$, so the point closest to $a$  on $[c_1, c_2]$ and the point
closest to $a$ on $[c_2, c_3]$ are the same and $w=Y(c_1, c_2, c_3)$.
By Lemma \ref{branch}, we have $2\dist(a, [c_1, c_3])=d(a,c_1)+d(a,c_3)-d(c_1,c_3)$.
Thus, $d(a, w)=\dist(a, [c_1, c_3])<{\epsilon}/{2}<\epsilon$.
So, $w=Y(c_1, c_2, c_3)$ has distance $<\epsilon$ from $a$.
Finally, for $i=1, 2, 3$ we have
\[
d(w,c_i) \geq d(a, c_i)-d(a,w)>s-\epsilon- \epsilon/2 >s-4\epsilon>h  \text{.}
\]
Therefore $b=Y(c_1, c_2, c_3)$ satisfies the desired conditions.
\end{proof}

The next lemma shows that any branch at any point in a model of $\rbRT_r$
must have maximum possible height, with that height realized by a point having distance $r$ from $p$.  This also
implies every model of $\rbRT_r$ has radius equal to $r$.

\begin{lemma}
\label{branches of sufficient height}
Let $\M\models \rbRT_r$ and let $a\in M$.
In any branch $\beta$ at $a$, there exists at least one point $b$ such that $d(p, b)=r$. 
\end{lemma}
\begin{proof}
Let $a\in M$ and $\delta=r-d(p, a)$, and let $\beta$ be a branch at $a$. 

We first assume\footnote{Here and in the final paragraph we fix a small mistake in the published version of this paper.}
$p\notin \beta$ and $a \neq p$.  Thus
$a$ is not an endpoint in $M$, so we have $\delta>0$ by Remark \ref{endpoints at distance r}.
In what follows, by iterating the use of Lemma \ref{branching point within epsilon},
we build a sequence $b_1, b_2, b_3,\dots$ of points in $\beta$
so that for any $k\in \N$, $[p, a, b_1, \dots, b_k]$ is a piecewise segment, and
$|r-d(p, b_i)|=r-d(p, b_i)<{\delta}/{2^i}$ .

First we choose $b_1$.  Let $c\in \beta$, so $d(a,c)>0$. 
If $c$ is such that $r-d(p, c)< {\delta}/{2}$, then
let $b_1=c$. Since $c$ and $p$ are in different branches at $a$, it follows that
$d(a,c) \leq \delta$. If $r-d(p,c)\geq {\delta}/{2}$, then let $0<\epsilon<{d(a, c)}/{2}$.
Use Lemma \ref{branching point within epsilon} to find $c'$ so that $d(c, c')<\epsilon$, 
and so that there are at least 3 branches at $c'$ with height at least
$(1-{d(a,c)}/{(2\delta)})(r-d(p, c))$.  So $|d(a, c)-d(a, c')|<\epsilon$, and since 
$\epsilon<{d(a, c)}/{2}$, we know $c'\in \beta$.  Then find a point $b_1$
on a branch at $c'$ other than the one containing $a$. 
This makes $[p, a, c', b_1]$ a piecewise segment.  Select $b_1$ so that 
\begin{align*}
 d(c', b_1) & =(1-\frac{d(a,c)}{\delta})(r-d(p,c)) \geq (1-\frac{d(a,c)}{\delta})\frac{\delta}{2}=\frac{\delta}{2}-\frac{d(a,c)}{2} \text{.} \\
\tag*{Then}  d(a, b_1) & =d(a, c')+d(c', b_1) >(d(a, c)-\epsilon)+d(c', b_1) \\
&>d(a,c)-\frac{d(a,c)}{2}+\frac{\delta}{2}-\frac{d(a,c)}{2} =\frac{\delta}{2} \text{.}
\end{align*}
Thus $d(a, b_1)>{\delta}/{2}$. It follows that
\begin{align*}
r-d(p, b_1) &=r-(d(p, a)+d(a, b_1))=(r-d(p,a))-d(a, b_1) \\
&=\delta-d(a, b_1)<\delta-\frac{\delta}{2}=\frac{\delta}{2}  \text{.}
\end{align*}

Once we have $b_1,\dots,b_i$, we proceed in a manner analogous to what was done above
to find $b_{i+1}$ so that $|r-d(p, b_i)|=r-d(p, b_i)<{\delta}/{2^i}$. Then for any $j>i$
we have 
\[
d(b_i, b_j)=d(p, b_j)-d(p, b_i)<r-(r-\frac{\delta}{2^i})=\frac{\delta}{2^i} \text{.}
\]
It follows that $(b_n)$ is a Cauchy sequence.  Let $b$ be the limit of $(b_i)$.  

Claim: $b_i\in [p, b]$ for all $i\in \N$. 

If not, let $i\in \N$ be such that $b_i\notin [p, b]$.
Then for all $j\geq i$, $b_j\notin [p, b]$, because $[p, b_i, b_j]$ is a piecewise segment. 
Let $j>i$ and let $q$ be the closest point to $b$ on $[p, b_i, b_j]$, {ie}, $q=Y(p, b_j, b)$. 
Either $q\in [p, b_i]$ or $q\in [b_i, b_j]$.  If $q\in [b_i, b_j]$, then $[p, q]=[p, b_i, q]$ 
implying $b_i\in [p, b]=[p, b_i, q, b]$, a contradiction. 
Thus, $q\in [p, b_i]$, and $[p, q, b_i, b_j]$ is a piecewise segment.
Using these facts and our choice of $q$ we conclude
\begin{align*}
 d(b, b_j) & =d(b, q)+d(q, b_j)=d(b, q)+d(q, b_i)+d(b_i, b_j) \\ 
 &\geq d(b, q)+d(q, b_i)=d(b, b_i) \text{.}
\end{align*}
We have demonstrated that  $d(b, b_j)\geq d(b, b_i)>0$ for all $j>i$, contradicting
that $b$ is the limit of the sequence $(b_i)$.  This proves the Claim.

Because $b_i\in [p, b]$, we know $r\geq d(p, b)\geq d(p, b_i)>r-{\delta}/{2^i}$  for any $i\in \N$.
Therefore, $d(p, b)=r$. And $b\in \beta$, since otherwise $b$ would be on a different branch
at $a$ than $b_1$. That would make $[b_1, a, b]$ a piecewise segment, contradicting
that $[p,b]=[p, a, b_1, b]$ is a piecewise segment.

Finally, we deal with the case where $p\in \beta \cup \{ a \}$. Since $\beta$ is open, we may use Lemma \ref{branching point within epsilon}
to find a point $a' \in \beta \setminus \{ p \}$ such that there are 3 branches at $a'$.
Select a branch $\beta'$ at $a'$ so that $p\notin \beta'$ and $a\notin \beta'$. 
The latter guarantees $\beta'\subseteq \beta$. 
Applying the first part of this proof to $a'$ and $\beta'$ yields a point $b \in \beta'$ such that
$d(p,b) = r$.
\end{proof}

The following theorem shows that complete richly branching pointed $\R$-trees
with radius $r$ form an elementary class.

\begin{theorem}
\label{rbrt}
The models of $\rbRT_r$ are exactly the complete, richly branching $\R$-trees with radius $r$.
\end{theorem}
\begin{proof}
Let $(M,d,p)$ be a complete, richly branching pointed $\R$-tree with radius $r$ and let $\M$
be the corresponding $L_r$-structure. Clearly $\M\models \RT_r$, and it remains
to verify that $\phi^\M=0$.
Let $a\in M$. If $d(p, a)=r$, then let $c_1=c_2=c_3=a$
and note that these witness $\psi^\M(a)=0$. 
So we assume $d(p, a)<r$, and let $\epsilon >0$  be such that $\epsilon/3<r-d(p,a)$.
Since $M$ is richly branching, the set $B$ from Definition \ref{richlybranching}
is dense in $M$.  So, there exists $b\in B$ with $d(a,b)<\epsilon/3$.
Then there exist $c_1, c_2, c_3\in M$ such that each distance $d(b,c_i)$ is $> r-d(b, p) - \epsilon/3$, 
and each of $c_1, c_2, c_3$ is on a different branch at $b$.  By moving these points
closer to $b$ along the geodesic segments if necessary, we may assume
$|d(b,c_i) - (r - d(b,p)| < \epsilon/3$ for each $i=1,2,3$.
The triangle inequality in $M$ implies $|d(a, c_i)-d(b, c_i)|<\epsilon/3$ for
all $i=1, 2, 3$, and $|(r-d(p, a))-(r-d(p, b))| < \epsilon/3$.  The triangle inequality
in $\R$ yields $|d(a,c_i)-(r-d(p,a))|<\epsilon$ for each $i=1, 2, 3$.

Next we fix $i \neq j$ in $\{1,2,3\}$ and note that
$d(c_i, c_j)=d(c_i, b)+d(c_j, b)$ 
because $c_1, c_2, c_3$ are on different branches at $b$.
Since we have $d(a,b) < \epsilon/3$, the triangle inequality yields
$|d(a,c_i)+d(a,c_j)-d(c_i, c_j)|< {2\epsilon}/{3}$.

So, for each $\epsilon$, there are $c_1, c_2$ and $c_3$ making
\[ 
\max_{i=1, 2, 3}\{|d(a,c_i)-(r-d(p,a))|\} < \epsilon
\]
\[
\tag*{and} \max_{1\leq i<j \leq 3}\{ |(d(a,c_i)+d(a,c_j))-d(c_i,c_j)| \}<\epsilon \text{.}
\]
Thus, $\psi(a)^\M=0$, and since $a$ was arbitrary, we conclude $\phi^\M=0$.

For the converse direction, assume $\M\models \rbRT_r$, so $(M, d, p)$ is a complete, 
pointed $\R$-tree. Let $a\in M$ and $\epsilon>0$.
By Lemma \ref{branching point within epsilon}
we may find $b\in M$ so that $d(a, b)<\epsilon$ and there are at least 3
distinct branches at $b$.
By Lemma \ref{branches of sufficient height}
each branch at $b$ contains a point with distance $r$ from $p$,
so the height of each branch is at least $r-d(p, b)$. It also follows from \ref{branches
of sufficient height} that $M$ has radius $r$.
Since $a\in M$ and $\epsilon>0$ were arbitrary, we conclude the
set $B$ of points with at least 3 branches of height $\geq r-d(p, b)$
is dense in $M$.  Therefore, $M$ is richly branching.
\end{proof}

Next, we turn to a series of lemmas that are needed for our proof that the theory 
$\rbRT_r$ is the model companion of $\RT_r$.  (See Theorem \ref{modcomp} below.)  

\begin{lemma}\label{ec implies richly branching}
Every existentially closed model of $\RT_r$ is a model of $\rbRT_r$.
\end{lemma}
\begin{proof}
Let $\M \models \RT_r$ be existentially closed with underlying $\R$-tree $(M,d,p)$.
Consider $a \in M$. If $r-d(p,a)=0$, then ${\psi}^{\M}(a)=0$
is true: make $y_1=y_2=y_3=a$.
When $r-d(p,a)>0$,  using Lemma \ref{chiswell add to points}
 we may construct an extension $\n \models \RT_r$ of $\M$
with underlying $\R$-tree $(N,d,p)$ such that in $N$ there exist $c_1, c_2$
and $c_3$, each on a different branch at $a$ and each with distance $r-d(p,a)$ from
$a$.  It follows that ${\psi}^{\n}(a)=0$.
Since $\M$ is existentially closed and $\psi$ is an $\inf$-formula, we have
${\psi}^{\n}(a)={\psi}^{\M}(a)=0$.
Our choice of $a \in M$ was arbitrary, so $\phi^\M=0$.
It follows that $\M$ is a model of $\rbRT_r$.
\end{proof}

The next lemma connects $\kappa$-saturation to the number
of branches at the points whose distance from the base point is $<r$, in a richly branching $\R$-tree of radius $r$.
(Recall from Remark \ref{endpoints at distance r} that every point at distance $=r$ from the base point 
in such an $\R$-tree is an endpoint.)  The converse of this lemma is also true, and the resulting 
characterization of $\kappa$-saturated models of $\rbRT_r$ is Theorem \ref{kappasaturated=kappabranching}
below, which will be proved once we have shown that $\rbRT_r$ admits quantifier elimination.

\begin{lemma}
\label{saturatedrbrt}
Let $\M\models \rbRT_r$. Let $\kappa$ be an infinite cardinal.
If $\M$ is $\kappa$-saturated, then $(M,d,p)$ has at least $\kappa$-many branches at 
every point $a$ such that $d(p,a)<r$.
\end{lemma}
\begin{proof}
We begin with the case where $\kappa=\omega$.  Arguing by contradiction,
we assume there exists an $\omega$-saturated model $\M$ of $\rbRT_r$ containing
a point $a$ with $d(p,a)<r$ such that the degree of $a$ is finite. 
Let  $\beta_1,\dots,\beta_n$ be the distinct branches at $a$.

For each $i=1,\dots,n$
we can find $b_i\in \beta_i$ so that $d(a, b_i)=r-d(p,a)$ as follows.
When $p$ is not in the branch $\beta_i$ we let $b_i$ be a point in $\beta_i$ with $d(p, b_i)=r$.
The existence of such a point is guaranteed by Lemma \ref{branches of sufficient height}.
In this case, either $p=a$, or $p$ and $b_i$ are on different branches
at $a$, and it follows in both instances that $d(a, b_i)=r-d(p,a)$.
When the base point $p$ is in branch $\beta_i$, we select $x\in [a, p]$ distinct
from $a$ and $p$ such that there are at least 3 branches at $x$.
Then $d(p, x)<d(p, a)$, so $r-d(p, x)>r-d(p, a)$.
On a branch at $x$ that contains neither $p$ nor $a$, by Lemma \ref{branches
of sufficient height}, there is a point $y$ so that $d(p, y)=r$. 
Note that $(a, y]\subseteq \beta_i$.
Since $p$ and $y$ are on different branches at $x$, we know $d(p, x)+d(x,y)=d(p, y)=r$,
so $d(x, y)=r-d(p, x)$.
Then  $$d(a, y)=d(a, x)+d(x, y)=d(a, x)+(r-d(p, x))> d(a, x)+(r-d(p, a))>r-d(p,a)$$
because $a$ and $y$ are on different branches at $x$.
Since $d(a, y)>r-d(p, a)$ we may take the point on $(a,y]$ with distance $r-d(p,a)$
from $a$ to be our $b_i$.

Next, let $r-d(p, a)>\epsilon>0$. Since $M$ is a richly branching $\R$-tree,
there exists $a'$ in branch $\beta_1$ such that  $0<d(a, a')< {\epsilon}/{2}$ and there are 
at least 3 branches at $a'$. Choose a branch at $a'$ that does not contain $a$
and does not contain $b_1$. Using Lemma  \ref{branches of sufficient height}, find $c$ 
in this branch at $a'$ so that $d(a', c)=r-d(p, a')$.
Note that by Lemma \ref{branch} we know $[b_1, a', c]$
and $[a, a', c]$ are both piecewise segments, as is $[b_i, a, a', c]$ for any $i=2,\dots,n$.

Using these facts and the fact that $d(a, a')<{\epsilon}/{2}$,
it is now straightforward to show that
\begin{align*} 
|d(a, c)-(r-d(p, a))|&<\epsilon\\
\tag*{and} |d(b_i, c)-(d(b_i, a)+d(a, c))|&<\epsilon \text{.}
\end{align*}

Let $\Sigma$ be the set of all conditions
\begin{align*}
|d(a, x)-(r-d(p,a))| & \leq 1/k \\
\tag*{and}  |d(b_i, x)-\left(d(b_i, a)+d(a, x)\right)| & \leq 1/k 
\end{align*}
for $k \in \N$ and $i=1,\dots,n$.

Since $\epsilon$ was arbitrary, we have that $\Sigma$
is finitely satisfiable.  Since there are only finitely many parameters in 
$\Sigma$, and $\M$ is $\omega$-saturated, there must exist $b\in M$ that satisfies
all of these conditions simultaneously. For this $b$ we have $d(a, b)=r-d(p, a)$ and
$d(b_i, b)=d(b_i, a)+d(a, b)$, meaning $b$ is in a different branch at $a$
from any of the $b_i$.  This contradicts that each branch at $a$ was represented
by one of $b_1,\dots, b_n$.  Therefore there must be at least $\omega$-many branches at
each point in $M$
when $\M$ is $\omega$-saturated.

Now, let $\kappa>\omega$ and assume $\M$ is $\kappa$-saturated. 
Let $a\in M$ with $d(p,a)<r$.
Assume toward a contradiction that there are exactly $\alpha$-many distinct branches 
at $a$ where $\alpha<\kappa$.  Index the branches at $a$ by $i<\alpha$,
and by Lemma \ref{branches of sufficient height} on each of these branches
designate a point $b_i$ such that $d(a, b_i)=r-d(p,a)$. 
Let $A=\{a\}\cup \{b_i\mid i<\alpha\}$
and note that $A$ has cardinality less than $\kappa$. Define
$$\Sigma=\{|d(a, x)-(r-d(p,a))|=0\}\cup\{|d(b_i, x)-\left(d(b_i, a)+d(a, x)\right)|=0\mid i<\alpha\} \text{.} $$
It is straightforward to show $\Sigma$ is finitely satisfiable, since by the first part
of this proof there must be infinitely many branches at $a$.
By $\kappa$-saturation there exists $b\in M$ that satisfies all of these conditions.
Using Lemma \ref{branch} we conclude that $b$ is on a different branch out of $a$ from each $b_i$,
contradicting the assumption that every branch at $a$ was represented by one of the $b_i$.
\end{proof}

\begin{lemma}
\label{embed}
 Let $\kappa$ be an infinite cardinal. Assume $\M$ is a $\kappa$-saturated model of 
 $\rbRT_r$ with underlying $\R$-tree $(M,d,p)$.
Let $K$ be a non-empty, finitely spanned $\R$-tree. Designate a base point $q$ in $K$.
For any $e\in M$ such that $r-d(p, e)\geq \sup\{d(q, x)\mid x\in K\}$
and any collection $\{\beta_i \mid i< \alpha\}$ of branches at $e$ with cardinality
$\alpha<\kappa$, there exists an isometric embedding $f$ of $K$
into $M$ such that $f(q)=e$ and $f(K)\cap \beta_i=\emptyset$ for all $i<\alpha$.
\end{lemma}
\begin{proof}
By Lemma \ref{saturatedrbrt},
$(M,d,p)$ has at least $\kappa$-many branches at each point $a$ satisfying $d(p,a)<r$. 
By Lemma \ref{mingenset}, there is
a minimal set that spans $K$, namely, the set of endpoints of $K$.
Proceed by induction on the size of this minimal spanning set,
building up the embedding at each step using the fact that there are $\kappa$-many 
branches of sufficient height at every interior point.
The restriction that $r-d(p, e)\geq \sup\{d(q, x)\mid x\in K\}$
keeps the image of the embedding inside $M$. 
\end{proof}

\begin{theorem}
\label{modcomp}
The $L_r$-theory $\rbRT_r$ is the model companion of $\RT_r$.
\end{theorem}
\begin{proof}
Since $\RT_r$ is an inductive theory, by Lemma \ref{axiomatizeec} it suffices
to show that the models of $\rbRT_r$ are exactly the existentially closed models of $\RT_r$.
By Lemma \ref{ec implies richly branching} we know every existentially closed model
of $\RT_r$ is a model of $\rbRT_r$.
It remains to show that every model of $\rbRT_r$ is an existentially closed model
of $\RT_r$.  Let $\M\models \rbRT_r$. Let $\n\models \RT_r$ be an
extension of $\M$. We may assume $\M$ and $\n$ are
$\omega_1$-saturated. (This is because we may consider the structure
$\big(\M,\n, \iota \big)$ where $\iota$ is the embedding
from $\M$ to $\n$, and take an $\omega_1$-saturated elementary extension of
that structure. If we can verify the definition of existentially
closed in that setting, it will be true of $\M$ and $\n$.) Let $(M,d,p)$ and $(N,d,p)$
be the underlying pointed $\R$-trees for $\M$ and $\n$ respectively.

Let $a=a_1,\dots,a_k\in M$, and without loss of generality assume base point $p$ is
among $a_1,\dots,a_k$.
We claim that for any $b_1,\dots,b_l\in N$, there exist $c_1,\dots,c_l\in M$
so that $d(b_i, b_j)=d(c_i, c_j)$ for all $i,j\in \{1,\dots,l\}$ and $d(a_i, b_j)=d(a_i, c_j)$
for all $i\in \{1,\dots,k\}$ and $j\in \{1,\dots,l\}$.\\
To prove this claim, let  $b_1,\dots,b_l\in N$, and
let $E_a\subseteq M$ be the subtree spanned by $a=a_1,\dots,a_k$.
Note that $E_a=\overline{E_a}$.
Define an equivalence relation on $b_1,\dots,b_l$ by: $b_i\sim b_j$
if $b_i$ and $b_j$ have the same closest point in $E_a$. Let $A_1,\dots,A_m$
be the equivalence classes of this equivalence relation, and for $v=1,\dots,m$
let $e_v$ be the unique closest point in $E_a$ common to the members of $A_v$.

For each $v=1,\dots,m$
let $K_v$ be the $\R$-tree spanned by $A_v\cup\{e_v\}$ in $N$, with base point $e_v$.
Note that each $K_v$ is closed and for $u\not=v$, $K_u\cap K_v=\emptyset$.
Since $(N, d, p)$ has radius $r$ and $p\in E_a$ and $e_v$ is the unique closest point to $x$ in $E_a$,
Lemma \ref{branch} gives
$d(p, e_v)+d(e_v, x)=d(p, x)\leq r$ for each $x\in K_v$. Therefore
$d(e_v, x)\leq r-d(p, e_v)$ for each $x\in K_v$. Thus, $\sup\{d(e_v, x)\mid x\in K_v\}\leq r-d(p, e_v)$.
Now, for each $v=1,\dots,m$, by Lemma \ref{embed}, there is an isometric embedding
$f_v\colon K_v\rightarrow M$ sending $e_v$ to $e_v$
such that $f_v(K_v)$ does not intersect $E_a$ except at $e_v$.  Note that $e_v$ is the
unique closest point in $E_a$ for every point in $f_v(K_v)$.

Let $f$ be the union of the functions $f_v$ for $v=1,\dots,m$. If $b_i$ and $b_j$
are both in $A_v$, then $d(b_i, b_j)=d(f_v(b_i), f_v(b_j))=d(f(b_i), f(b_j))$.
If $b_i$ and $b_j$ are in $A_{u}\not=A_{v}$ respectively, then using Lemma
\ref{differentclosestpoints}
\begin{align*}
d(b_i, b_j)&=d(b_i, e_u)+d(e_u, e_v)+d(e_v, b_j)\\
&= d(f_u(b_i), f_u(e_u))+d(e_u, e_v)+d(f_v(e_v), f_v(b_j))\\
&= d(f(b_i), e_u)+d(e_u, e_v)+d(e_v, f(b_j))\\
&=d(f(b_i), f(b_j)).
\end{align*}
Therefore the function $f$ is an isometric embedding from
$\displaystyle\bigcup_{v=1}^m K_v$ to $M$.\\
Let $c_j=f(b_j)$ for all $j\in\{1,\dots,l\}$.
Then clearly $d(b_i, b_j)=d(f(b_i), f(b_j))=d(c_i, c_j)$ for all $i,j\in \{1,\dots,l\}$.
Now let $i\in\{1,\dots,k\}$ and $j\in \{1,\dots,l\}$.  Let $e_v$ be the closest point to $b_i$
in $E_a$.  Then
\begin{align*}
d(a_i,b_j)&=d(a_i, e_v)+d(e_v, b_j)\\
&=d(a_i, e_v)+d(f(e_v),f(b_i))\\
&=d(a_i, e_v)+d(e_v, c_j)=d(a_i, c_j).
\end{align*}
Thus, the claim is true.

The values of quantifier free
formulas in $\M$ are determined by distances in $M$. (This can be shown
using induction on the definition of quantifier free formula, since connectives
are continuous functions on atomic formulas, which in this case are all of the
form $d(t_1,t_2)$ for terms $t_1, t_2$.)
So, the preceding claim implies that for any quantifier free formula
$\phi(x_1,\dots,x_k, y_1,\dots,y_l)$
and any $b_1,\dots,b_l\in \n$,
there exist $c_1,\dots,c_l\in \M$ so that
$$\phi(a_1,\dots,a_k,b_1,\dots,b_l)^\n=\phi(a_1,\dots,a_k,c_1,\dots,c_l)^\n=\phi(a_1,\dots,a_k,c_1,\dots,c_l)^\M \text{.} $$
Then standard arguments about infima imply
$$\inf_{y_1}\dots\inf_{y_l}\phi(a_1,\dots,a_k, y_1,\dots,y_l)^\M=\inf_{y_1}\dots\inf_{y_l}\phi(a_1,\dots,a_k, y_1,\dots,y_l)^\n \text{.} $$
Therefore, $\M$ is an existentially closed model of $\rbRT_r$.
\end{proof}

\section{Properties of the theory \texorpdfstring{$\rbRT_r$}{rbRTr}}
\label{propertiesrbrt}
In this section we show that $\rbRT_r$ has
quantifier elimination and is complete and stable, but not superstable.  We characterize types, and show that
the space of 2-types over the empty set has metric density $2^\omega$. We also characterize definable closure 
and algebraic closure in models of $\rbRT_r$. 

\begin{notation}
If $A = A_1,\dots,A_n$ is a finite sequence of subsets of a model, we will use the shorthand notation
$A_1 \dots A_n$ for the union $A_1 \cup \dots \cup A_n$.  If some $A_j$ has a single element, we write
the element instead of the set; for example $ABC$ stands for $A \cup B \cup C$ and $Ap$ stands for $A \cup \{p\}$.  
This notation will be used mainly when we are considering types over sets of parameters.
\end{notation}

\begin{lemma}
\label{qe}
The $L_r$-theory $\rbRT_r$ has quantifier elimination.
\end{lemma}
\begin{proof}
By Theorem \ref{modcomp}, Theorem \ref{amalg} and Proposition \ref{modelcompanionplusamalg}.
\end{proof}

\begin{corollary}
The $L_r$-theory $\rbRT_r$ is complete.
\end{corollary}
\begin{proof}
In any model of $\rbRT_r$ we may embed the structure consisting of just the base point.
This fact together with quantifier elimination implies that $\rbRT_r$ is complete.
\end{proof}

We now turn our attention to a discussion of types and type spaces.
Let $\M\models \rbRT_r$. Recall that if $b=b_1,\dots,b_n$ is tuple of elements in $\M$, then $\tp_\M(b/A)$ 
is the complete $n$-type of $b$ over $A$ in $\M$.
The space of all $n$-types over $\emptyset$ in models of a theory $T$ is denoted $S_n(T)$. The space
of $n$-types over $A$ is denoted $S_n(T_A)$ or $S_n(A)$ when the theory $T$ is
clear from context. If $q$ is an $n$-type and $a$ is an  $n$-tuple from $M$ such that $\M\models q(a)$ we say $a$ 
\emph{realizes} $q$ in $\M$, and write $a\models q$. If $\phi(x_1,\dots, x_n)$ is an $L_r$-formula,
we write $\phi(x_1,\dots, x_n)^q$ for the value of this formula specified by the type $q$. 

The following lemma gives criteria for when two $n$-tuples have the same type.
Recall that for a subset $A$ of a given $\R$-tree $M$
$$E_A=\bigcup\{[a_1,a_2]\mid a_1, a_2\in A\}$$
is the smallest subtree of $M$ containing $A$.
The closure $\overline{E_A}$ is the smallest closed subtree containing $A$.
\begin{lemma}
\label{determinetypes}
Let $\M\models \rbRT_r$. Fix $A\subseteq M$.
\begin{enumerate}
\item Let $b,c \in M$. Then $\tp_\M(b/A)=\tp_\M(c/A)$ if and only if
$b$ and $c$ have the same unique closest point $e\in \overline{E_{Ap}}$ and
$d(b,e)=d(c,e)$.
\item Let $b=(b_1,\dots,b_n)$ and $c=(c_1,\dots,c_n)$ be tuples in $M$. 
Then $\tp_\M(b/A)=\tp_\M(c/A)$ if and only if $\tp_\M(b_i/A)=\tp_\M(c_i/A)$
for all $1 \leq i \leq n$ and $d(b_i,b_j)=d(c_i,c_j)$ for all $1 \leq i,j \leq n$.
\item Let $b=(b_1,\dots,b_n)$ and $c=(c_1,\dots,c_n)$ be tuples in $M$. Let 
$E_b$ be the subtree of $M$ spanned by $Apb$, and similarly define $E_c$.
Then $\tp_\M(b/A)=\tp_\M(c/A)$ if and only if there is an isometry $f\colon E_b\rightarrow E_c$
that fixes each element of $Ap$ (or equivalently of $\overline{E_{Ap}}$) and satisfies $f(b_i)=c_i$ for all $i=1,\dots,n$.
\end{enumerate}
\end{lemma}
\begin{proof} 
To show (1), first assume $\tp_\M(b/A)=\tp_\M(c/A)$.  By Lemma \ref{0hyp}(3) this implies
$d(b,a)=d(c,a)$ for all $a\in \overline{E_{Ap}}$, which means $b$ and $c$ must
have the same unique closest point  $e\in \overline{E_{Ap}}$.
Moreover, $d(b, e)$ must equal $d(c,e)$.

For the other direction, assume $b$ and $c$ have the same
unique closest point  $e\in  \overline{E_{Ap}}$ and that $d(b, e)=d(c,e)$.
Since $\rbRT_r$ has quantifier elimination,  and the values of quantifier-free formulas
are determined by the values of atomic formulas, it suffices to show
$d(a, b)=d(a,c)$ for all $a\in Ap$.  This follows easily
from our assumptions using Lemma \ref{branch}, since for any $a\in Ap$,
the point $e$ must be on both segments $[a, b]$ and $[a,c]$.

Statement (2) follows from part (1) and the fact that $\rbRT_r$ admits quantifier elimination.
Note that $\tp_\M(b_i/A)=\tp_\M(c_i/A)$
for all $1 \leq i \leq n$ if and only if $d(b_i, a)=d(c_i, a)$ for all $a\in Ap$. 
Again, we use that the values of quantifier-free formulas are determined
by the values of the atomic formulas. The atomic formulas involved are
all of the form $d(t_1, t_2)$ where $t_1$ and $t_2$ are variables
$x_1,\dots,x_n$ or parameters from $Ap$.

Statement (3) follows from (2) and Lemma \ref{0hyp}.
\end{proof}

With Lemma \ref{determinetypes} in hand, we can finish the characterization of $\kappa$-saturated $\R$-trees, 
completing the result promised in Section \ref{model companion}.

\begin{theorem}
\label{kappasaturated=kappabranching}
Let $\M\models \rbRT_r$.
Let $\kappa$ be an infinite cardinal.
Then $\M$ is $\kappa$-saturated if and only if $M$ has at least $\kappa$-many branches
at every point $a\in M$ where $d(p, a)<r$.
\end{theorem}
\begin{proof}
The forward direction is Lemma \ref{saturatedrbrt}.
Now, assume $M$ has $\kappa$-many
branches at every point $a\in M$ where $d(p, a)<r$.  
Let $A\subseteq M$ have cardinality $< \kappa$.

Let $q$ be a $1$-type over $A$.
By Lemma \ref{determinetypes}, this type is determined by a
closest point $e\in \overline{E_{Ap}}$ and a distance $0\leq s\leq r-d(p, e)$,
where $s$ is the distance between $\overline{E_{Ap}}$ and any realization of $q$.
Suppose $\beta$ is a branch in $M$ at $e$.
The complement of $\beta$ is path connected and closed, so if the complement 
contains all of $Ap$, then it contains $\overline{E_{Ap}}$.
Hence, if $\beta$ contains a point in $\overline{E_{Ap}}$, it must contain an element of $Ap$.

Because there are $\kappa>|Ap|$-many branches at $e$, 
there are branches at $e$ in $M$ that do not intersect $\overline{E_{Ap}}$.
Let $\beta$ be one of those branches.  Lemma \ref{branches of sufficient height} 
allows us to take $x \in \beta$ such that $d(p,x) = r$, and hence
$d(x,e) = r-d(p,e)$, since $p$ is not in $\beta$.  Since $0\leq s\leq r-d(p, e)$ we 
may take $b$ on the segment $[e,x]$ with $d(e,b)=s$.
Then $b$ satisfies the type $q$.  Therefore, all $1$-types over $A$ are realized 
in $\M$, implying that $\M$ is $\kappa$-saturated.
\end{proof}

Our next result gives a description of the type space $S_n(A)$ with its logic topology,
where $A$ is a subset of a model $\M$ of $\rbRT_r$.  If $q \in S_n(A)$, we take the free
variables used in $q$ to be the distinct variables $x_1,\dots,x_n$ and require that 
$X = \{x_1,\dots,x_n\}$ be disjoint from $Ap$.  Note that $q$ induces a pseudometric, which
we denote $\rho_q$, on $ApX$, by setting $\rho_q(t_1,t_2) = d(t_1,t_2)^q$ for each $t_1,t_2 \in ApX$.  That is, if
$(b_1,\dots,b_n)$ realizes $q(x_1,\dots,x_n)$ in some elementary extension $\n$ of $\M$, and we let
$\pi \colon ApX \to N$ be the evaluation map taking each $a \in Ap$ to itself and taking each $x_i$ to
$b_i$, then we are setting $\rho_q(t_1,t_2)$ equal to $d^\n(\pi(t_1),\pi(t_2))$.  Note that $\pi$ can be viewed
as the quotient map from the pseudometric space $(ApX,\rho_q)$ to the metric space $(Ap\{b_1,\dots,b_n\},d^\n)$, which is
$0$-hyperbolic since $N$ is an $\R$-tree.  It follows that $\rho_q$ is $0$-hyperbolic on $ApX$.
(See Definition \ref{defhyperbolic} and Lemma \ref{0hyp}.)

For each $q \in S_n(A)$, we regard $\rho_q$ as an element of the product space $[0,2r]^{(ApX)^2}$, to which we
give the product topology.  Let $\HM(ApX)$ denote the set of all $\rho \in [0,2r]^{(ApX)^2}$ such that $\rho$ defines
a $0$-hyperbolic pseudometric on $ApX$ extending $d^\M$ on $Ap$ and satisfying
$\rho(p,x_i) \leq r$ for all $i=1,\dots,n$.  This is a closed subset of $[0,r]^{(ApX)^2}$ and hence 
it is a compact Hausdorff space in the induced topology.

\begin{lemma}\label{n types over A}
Let $\M\models \rbRT_r$ and $A\subseteq M$.
\newline
(1) The image of the map $q \mapsto \rho_q$ defined above on $S_n(A)$ is $\HM(ApX)$. 
\newline
(2) The map $q \mapsto \rho_q$ is a homeomorphism between $S_n(A)$ with the logic topology and $\HM(ApX)$
with the induced topology.
\end{lemma}
\begin{proof}  Let $\M \models \rbRT_r$ and let $A \subseteq M$.

(1)  Suppose $\rho \in \HM(ApX)$.  Recall that the quotient metric space of $(ApX,\rho)$ is $0$-hyperbolic (see the comment 
just before Lemma \ref{triangle property}) and the quotient map from $(ApX,\rho)$ to that metric space is isometric.
Therefore, by Lemma \ref{0hyp}(2) there is an isometric mapping $f$ from $(ApX,\rho)$ into an $\R$-tree $(N,d)$.  We
may assume that $d(f(p),u) \leq r$ for all $u \in N$ and that $(N,d)$ is complete, and hence that $\n = (N,d,f(p)) \models \RT_r$. 
By Theorem \ref{modcomp} we may assume that $\n$ is a $\kappa$-saturated and strongly $\kappa$-homogeneous model of $\rbRT_r$, where
$\kappa$ is an infinite cardinal $> |A|$.  Since $\rbRT_r$ admits QE we may assume that $Ap \subseteq N$ and that $f$ is the identity on
$Ap$.  Now let $b = (b_1,\dots,b_n) = (f(x_1),\dots,f(x_n)) \in N^n$, and let $q(x_1,\dots,x_n) = tp(b/A)$ in $\n$.  It is easy to see that
$\rho = \rho_q$, and hence $\rho$ is in the image of the map $q \mapsto \rho_q$ as desired.

(2)  Quantifier elimination for $\rbRT_r$ implies that the map $q \mapsto \rho_q$ is 1-1 and continuous for the logic topology on 
$S_n(A)$ and the induced topology on $\HM(ApX)$, and this map is surjective by part (1).  
Both topologies are compact Hausdorff, and hence the map must be a
homeomorphism.
\end{proof}

Next, we give some results about the type spaces as metric spaces.
First, a reminder of how the induced metric on types is defined.

Let $\M\models \rbRT_r$ be such that every type in $S_n(\rbRT_r))$ is realized in $\M$ for each $n\geq 1$.
As defined in 
Ben Yaacov, Berenstein, Henson and Usvyatsov \cite{N}, 
the $d$-metric on $n$-types over the empty set is:
$$d(q,q')=\inf_{b\models q,\ b'\models q'}\ \max_{i=1,\dots,n}\  d^\M(b_i, b'_i)$$
where $b=b_1,\dots,b_n$ and $b'=b'_1,\dots,b'_n$ are tuples in $\M$.
Note that since $\M$ realizes all $2n$-types, the infimum in the definition is actually realized. This definition 
can be extended in the obvious way to spaces of types over parameters.

First we consider the distance between $1$-types over $A$ for the theory $\rbRT_r$.

\begin{lemma} 
\label{howmanytypes}
 Let $\M \models \rbRT_r$ and $A \subseteq M$.
\begin{enumerate}
\item As a metric space, $S_1(\rbRT_r)=S_1(\emptyset)$  is isometric to $[0, r]$.
\item More generally,  $S_1(A)$ is 
in bijective correspondence with the set of ordered pairs
$$\{(e,s)\mid e\in \overline{E_{Ap}} \text{ and } s\in [0,r-d(p,e)]\subseteq \R\} \text{.}$$
Indeed, $q \in S_1(A)$ is associated to $(e,s)$ exactly when $e$ is the unique
closest point in $\overline{E_{Ap}}$ to any realization $b$ of $q$ (in an elementary
extension $\n$ of $\M$) and $s = d^\n(e,b)$; we will denote this distance $s$ by
$d(e,x)^q$.
\item Given $q,q'\in S_1(A)$ let $e_q$, $e_{q'}$ be the unique
closest points in $\overline{E_{Ap}}$ to realizations of $q$, $q'$ respectively.
\begin{enumerate}
\item  If $e_q\not=e_{q'}$, then
$d(q,q')= d(e_q,x)^q+d(e_q, e_{q'})+d(e_{q'},x)^{q'}$.
\item If $e_q=e_{q'}$, then $d(q,q')=|d(p,x)^q-d(p,x)^{q'}|$.
\end{enumerate}
\item  The metric space $(S_1(A),d)$ is an $\R$-tree, isometric to the $\R$-tree
obtained from $\overline{E_{Ap}}$ by amalgamating an additional branch $\beta_e$ at $e$
that is isometric to the interval $(0,r-d(p,e)]$, for each $e \in \overline{E_{Ap}}$.
\end{enumerate}
\end{lemma}
\begin{proof}  Let $\M \models \rbRT_r$ and $A \subseteq M$.  

(1) Lemma \ref{n types over A} gives a bijection $f$ from $S_1(\emptyset) = S_1(\rbRT_r)$
onto $[0,r]$, in which each type $q(x)$ is mapped to $d(p,x)^q$.  We show that $f$ is isometric.
Given two 1-types $q,q' \in S_1(\rbRT_r)$
a configuration minimizing the distance between realizations $b\models q$
and $b'\models q'$ is the one where $b,b'$ and $p$ are arranged along a piecewise segment.
That configuration makes
$d(b, b')=|d(p, b)-d(p, b')|$ which is the least possible value of $d(b,b')$. by the triangle inequality.
Thus, $d(q,q')=d(b, b')=|d(p, b)-d(p, b')|=|f(q)-f(q')|$, showing that $f$ is an isometry.

(2) This is a consequence of Lemmas \ref{determinetypes}, \ref{n types over A}, and \ref{0hyp}(3).
Lemma \ref{determinetypes}(1) shows that the pair $(e,s)$ associated to $q$ in (2) is indeed
determined by $q$.  
A description of how to find $b$ realizing a type $q$ associated to $(e,s)$ 
in a sufficiently saturated elementary extension of $\M$ is given in the last
paragraph of the proof of Theorem \ref{kappasaturated=kappabranching}.

(3) Let $q,q' \in S_1(A)$.  
First assume $e_q\not=e_q'$. Then for any realizations $b\models q$ and $b'\models q'$,
$[b, e_q, e_{q'}, b']$ is a piecewise segment by Lemma \ref{differentclosestpoints}, and the conclusion
follows.  If $e_q=e_{q'}$, then as in the proof of (1), a minimizing configuration will have the three
points $b\models q$ and $b'\models q'$ and $p$ along a geodesic segment, and the conclusion follows.

(4) This is immediate from (2) and (3).
\end{proof}

Note that part (2) in Lemma \ref{howmanytypes}  
gives a description of $S_1(A)$ that is different from the one in Lemma \ref{n types over A}.
It is possible to give an analogous description of $S_n(A)$ when $n > 1$.  
Let $x_1,\dots,x_n$ be distinct variables, with $X = \{x_1,\dots,x_n\}$ disjoint from $\overline{E_{Ap}}$.
Suppose we are given $q(x_1,\dots,x_n) \in S_n(A)$  and a realization $(b_1,\dots,b_n)$ of $q$ in
some model $\M$ of $\rbRT_r$ containing $A$.  Let $\pi$ be the function on the set
$\overline{E_{Ap}}X$ defined by $\pi(e)=e$ for $e \in \overline{E_{Ap}}$ and $\pi(x_i)=b_i$ for $i = 1,\dots,n$.
According to parts (1) and (2) of Lemma \ref{determinetypes}, the type $q$ is determined by the 
pseudometric $\eta = \eta_q$ on $\overline{E_{Ap}}X$  that is defined by setting
$\eta(t_1,t_2) = d^\M(\pi(t_1),\pi(t_2))$ for all $t_1,t_2 \in \overline{E_{Ap}}X$.  The key properties
of $\eta$ are the following:
\newline
(a) $\eta$ agrees with $d^\M$ on $\overline{E_{Ap}}$;
\newline
(b) $\eta(p,x_i) \leq r$ for all $i=1,\dots,n$; 
\newline
(c) $\eta(e,x_i) = \eta(e,e_i) + \eta(e_i,x_i)$ for all $e \in \overline{E_{Ap}}$, $i=1,\dots,n$; and
\newline
(d) $\eta$ is $0$-hyperbolic.
\newline
For every such pseudometric $\eta$ there is an $n$-type $q \in S_n(A)$ for which $\eta = \eta_q$.  
Indeed, by Lemma \ref{0hyp}, such a pseudometric $\eta$ is determined by its restriction to the set 
$ApX$, which is an element of $\HM(ApX)$ defined above.  Thus it corresponds to a type in $S_n(A)$ 
by Lemma \ref{n types over A}.

For $n>1$ it seems to be very complex to give a precise description, as in the preceding lemma, of the metric on 
$S_n(A)$ over the theory $\rbRT_r$. Accordingly, we limit ourselves to giving (in the next result) a precise
statement of the metric density of the space of 2-types over $\emptyset$. Its proof illustrates
some of the ideas needed to understand these metric spaces more completely.

\begin{proposition}
The space of $2$-types over the empty set has metric density character equal to $2^\omega$.
\end{proposition}
\begin{proof} 
It is clear that $S_2(\rbRT_r)$ has cardinality $2^\omega$.  Hence it suffices to find, for some $\delta>0$, 
a set of $2^\omega$-many types in $S_2(\rbRT_r)$ that are pairwise at distance $\geq \delta$.

For each $s \in [0,r]$ let $q_s(x,y) \in S_2(\rbRT_r)$ contain the conditions $d(p,x)=d(p,y)=r$ and $d(x,y)=2s$.
Note that for every $s$ there is a unique such type.  If $(a_s,b_s)$ realizes $q_s$ in $\M \models \rbRT_r$ we let $Y_s = Y(p,a_s,b_s)$, 
and note that $d(a_s,Y_s)=d(b_s,Y_s)=s$ and $d(p,Y_s)=r-s$.

Suppose $0 \leq t < s \leq r$.  We will show that the distance between $q_s$ and $q_t$ is $\geq 2t$.  So if we take
$0 < \delta < r$, the types in  $\{ q_s \mid \delta \leq s \leq r \}$ are pairwise at distance $\geq 2\delta$, giving the
desired result.

So suppose for $0 \leq t < s \leq r$ we have $(a_s,b_s)$ realizing $q_s$ and $(a_t,b_t)$ realizing $q_t$, in a model $\M$ of $\rbRT_r$,
and that $d(a_s,a_t) \leq 2t$.  Letting $c = Y(p,a_s,a_t)$ we have $d(p,c) \geq r-t$, implying that $Y_s$ and $Y_t$ are
both in the geodesic segment $[p,c]$.  Then, since $Y_s \neq Y_t$, we have by Lemma \ref{differentclosestpoints} 
that $[b_s,Y_s,Y_t,b_t]$ is a piecewise segment, and hence
\[
d(b_s,b_t) = d(b_s,Y_s)+d(Y_s,Y_y)+d(Y_t,b_s) = s + d(Y_s,Y_t) + t > 2t \text{,}
\]
which completes the proof since the realizations of $q_s,q_t$ were arbitrary.
\end{proof}

\begin{remark*}
A similar argument shows that if $X$ is a metric $\R$-tree that contains an isometric copy of 
every 4-point 0-hyperbolic metric space of diameter at most r, 
then the density character of $X$ is at least $2^\omega$.  Likewise, if $a$ is a point different from $p$  in 
$\M \models \rbRT_r$, then the type space $S_1(\{a\})$ (relative to $\rbRT_r$) has density character exactly $2^\omega$.
\end{remark*}

Next, we show that the definable closure and the algebraic closure
of a set of parameters $A$ are the same, and are equal to the closed subtree spanned
by $Ap$.  

\begin{proposition}
\label{L: dcl and acl characterized}
Let $\M\models \rbRT_r$ with $(M,d,p)$ as its underlying $\R$-tree.
Let $A\subseteq M$ be a set of parameters. Then $\dcl(A)=\acl(A)=\overline{E_{Ap}}$.
\end{proposition}
\begin{proof}
If $a,b\in Ap$, then by Lemma \ref{definablemidpoints} any point in $[a,b]$ is in $\dcl(A)$. 
Then $E_{Ap}\subseteq \dcl(A)$ because $E_{Ap}$ is the union of all such geodesic segments.
Since $\dcl(A)$ is closed, $\overline{E_{Ap}}\subseteq \dcl(A)$. Combined with the fact that
$\dcl(A)\subseteq \acl(A)$, this gives $\overline{E_{Ap}}\subseteq \dcl(A)\subseteq \acl(A)$.
It remains to show $\acl(A)\subseteq \overline{E_{Ap}}$, which we do in the contrapositive. 

If $c\notin \overline{E_{Ap}}$, let $e$ be the unique closest point to $c$
in $\overline{E_{Ap}}$. This $e$ exists by Lemma \ref{!geo}.
Let $\beta$ be the branch at $e$ that contains $c$.
Using Lemma \ref{chiswell add to points}, construct
an extension $\n\models \rbRT_r$ of $\M$ which adds an infinite number of branches at $e$, 
each of which is isometric to $\beta$.   
By Lemma \ref{determinetypes}, on each of these branches is a realization of $\tp(c/A)$. The 
distance between any two such realizations is $2d(e, c)>0$. This gives us a non-compact set of 
realizations of $\tp(c/A)$. Therefore, $c\notin \acl(A)$. 
\end{proof}

To finish this section we show $\rbRT_r$ is stable, but not superstable ({ie}, it is strictly stable).

\begin{theorem}
\label{rbrtstability}
The theory $\rbRT_r$ is stable.
Indeed when $\kappa$ is an infinite cardinal, $\rbRT_r$ is $\kappa$-stable
if and only if $\kappa$ satisfies $\kappa^\omega = \kappa$.
\end{theorem}
\begin{proof}
Let $\kappa$ be an infinite cardinal.
Let $\M\models \rbRT_r$ be $\kappa^{+}$-saturated, with underlying $\R$-tree $(M,d,p)$.\\
First, assume $\kappa=\kappa^\omega$.
Let $|A|=\kappa$. Then
\[
|E_{Ap}|\leq |A\times A|2^\omega=\kappa^2 2^\omega
\leq \kappa^\omega 2^\omega=\kappa^\omega=\kappa \text{.}
\]
Thus, $|E_{Ap}|=\kappa$.
We count possible $1$-types using Lemma \ref{howmanytypes}, showing that
\[
|S_1(A)|\leq |\overline{E_{Ap}}\times [0,r]|=|\overline{E_{Ap}}|2^\omega\leq
|E_{Ap}|^\omega 2^\omega=\kappa^\omega 2^\omega=\kappa^\omega=\kappa \text{.}
\]
Thus $\rbRT_r$ is $\kappa$-stable.

For the other direction, assume $\kappa<\kappa^\omega$. We construct, via a tree construction, 
a subset $A$ of $M$ with $|A|=\kappa$ and
$|\overline{E_{Ap}}|=\kappa^\omega$.
At Step 1, we choose $\kappa$-many points $(a_i\mid i<\kappa)$ on distinct branches at $p$, 
each with distance ${r}/{4}$ from $p$.
We can do this since there are at least $\kappa$ branches of sufficient height at every point in $M$ by 
Theorem \ref{kappasaturated=kappabranching} and Lemma \ref{branches of sufficient height}. 
Note that we only need $\kappa$-saturation to guarantee $\kappa$-many branches.  
At Step 2, for each $a_{i}$ we choose $\kappa$-many points on distinct branches at $a_i$,
each with distance ${r}/{8}$ from $a_i$, and distance ${3r}/{8}$ from $p$. 
We index these points by $(a_{i,j}\mid i,j<\kappa)$.
At Step $n$ for $n\geq 2$, we have already designated points $a_{i_1, i_2,\dots,i_{n-1}}$ each
of which has distance
$\sum_{k=1}^{n-1} {r}/{2^{k+1}}$ from $p$. At each of these points choose $\kappa$-many
points on distinct branches at $a_{i_1, i_2,\dots,i_{n-1}}$,
each with distance ${r}/{2^{n+1}}$ from $a_{i_1, i_2,\dots,i_{n-1}}$, and distance $\sum_{k=1}^{n} {r}/{2^{k+1}}$ from $p$. 
We index these new points by $(a_{i_1, i_2,\dots,i_n}\mid i_1,\dots,i_n<\kappa)$. 
Let $A=\bigcup_{k=1}^\infty (a_{i_1,\dots,i_k}\mid i_1,\dots,i_k<\kappa)$. If we associate $p$
with the empty sequence, then the elements of $Ap$ are in 1-1 correspondence with $\kappa^{<\omega}$.
So, the cardinality of $A$ is $|\kappa^{<\omega}|=\kappa$.

If $\sigma$ and $\tau$ are non-empty finite sequences from $\kappa$, 
we have that $[p,a_\sigma,a_{\sigma,\tau}]$ is a 
piecewise segment in $M$.  (Here $\sigma,\tau$ denotes the concatenation of $\sigma$ and $\tau$.)
Moreover, for every such $\sigma$ and every distinct $i,j < \kappa$, the points $a_{\sigma,i}$, $a_{\sigma,j}$, and
$p$ are in distinct branches at $a_\sigma$.  

For each function $f\colon \omega \rightarrow \kappa$ with $f(0)=0$ we define a
sequence $(b^f_n)$ of elements of $A$ by setting $b^f_0=p$ and $b^f_n=a_{f(1),\dots,f(n)}$ for $n>0$.
For each $n \in \omega$ we have that $[b^f_0,\dots,b^f_n]$ is a piecewise segment in $M$.
Further, $b^f_k$ has distance ${r}/{2^{k+1}}$ from $b^f_{k-1}$ for every $k>1$,
making $(b^f_n)$ a Cauchy sequence. Since $M$ is complete, 
$b^f_n$ must converge to a limit, which we denote by $u_f$.  Note that $d(p,u_f) = r/2$; 
indeed, $d(b^f_n,u_f) = r/2^{n+1}$ for every $n$.

Let $f$ and $g$ be two such functions from $\omega$ to $\kappa$, and suppose 
$m>0$ is the first index at which $f(m)$ and $g(m)$ disagree.  An easy calculation shows that
$u_f$, $u_g$, and $p$ are in distinct branches at $b^f_{m-1} = b^g_{m-1}$ in $M$ and 
$d(u_f,u_g) = r/2^{m-1}$.  In particular, the map $f \mapsto u_f$ is injective, so
the set of all points $u_f$ has cardinality $\kappa^\omega$.

For each $u_f$ as above, let $\beta_f$ be the branch at $u_f$ in $M$ that contains $p$.
Evidently $A \subseteq \beta_f$; since $\beta_f \cup \{u_f\}$ is closed and path-connected, 
it must contain $\overline{E_{Ap}}$.  We choose a point $v_f \in M$ in a different branch at $u_f$
from $\beta_f$ such that $d(v_f,u_f) = r/2$.  Note that $u_f$ is the closest point to
$v_f$ in $\overline{E_{Ap}}$.

By Lemma \ref{determinetypes} the type of $v_f$ over $A$ is determined by $u_f$
and $d(v_f,u_f)$.  Further, by Lemma \ref{howmanytypes}(3), for distinct $f,g$ we have
\[
d(\tp(v_f/A),\tp(v_g/A)) = d(u_f,v_f) + d(u_f,u_g) + d(u_g,v_g) > r \text{.}
\]
Therefore, the metric density character of $S_1(A)$ is at least $\kappa^\omega$.  Since $A$
has cardinality $\kappa$ and $\kappa^\omega > \kappa$, this shows that the theory
$\rbRT_r$ is not $\kappa$-stable.
\end{proof}

\section{The independence relation for \texorpdfstring{$\rbRT_r$}{rbRTr}}

In this section we characterize the model theoretic independence relation of
$\rbRT_r$ and show that in models of $\rbRT_r$, types have canonical
bases that are easily-described sets of ordinary (not imaginary) elements.

Let $\kappa$ be a cardinal so that $\kappa=\kappa^\omega$ and $\kappa >2^\omega$.
In this section let $\U$ be a $\kappa$-universal domain for $\rbRT_r$.  (That is, $\U$ is
$\kappa$-saturated and $\kappa$-strongly homogeneous; see \
Ben Yaacov, Berenstein, Henson and Usvyatsov \cite[Definition 7.13]{N}
and the discussion preceding it.)
A subset of $\U$ is \emph{small} if its cardinality is $<\kappa$.
\begin{definition}
\label{starindependence}
Let $A,B$ and $C$ be small
subsets of $\U$.
Say A is ${}^*$-independent from B over C, denoted $A\ind[*]_C B$,
if and only if for all $a\in A$ we have $\dist(a,\overline{E_{BCp}})=\dist(a,\overline{E_{Cp}})$.
\end{definition}

\begin{lemma}
$A\ind[*]_C B$
if and only if for all $a\in A$ the closest point to $a$ in
$\overline{E_{BCp}}$ is the same as the closest point to $a$ in $\overline{E_{Cp}}$.
\end{lemma}
\begin{proof}
($\Rightarrow$)  Assume $A\ind[*]_C B$.
Take an arbitrary $a\in A$. Let $e_1$ be the unique closest point to $a$ in $\overline{E_{BCp}}$
and $e_2$ the unique closest point to $a$ in $\overline{E_{Cp}}$.
We assumed $\dist(a,\overline{E_{BCp}})=\dist(a,\overline{E_{Cp}})$, which implies $d(a, e_1)=d(a, e_2)$.
Since $e_2\in \overline{E_{Cp}}\subseteq \overline{E_{BCp}}$, we know $e_1\in [a,e_2]$ by Lemma \ref{!geo}.
Therefore, $e_1=e_2$. Since $a$ was arbitrary, we know this holds for all $a\in A$.
\newline
($\Leftarrow$)  Assume for all $a\in A$ the closest point to $a$ in $\overline{E_{BCp}}$
is the closest point to $a$ in $\overline{E_{Cp}}$. Then clearly
$\dist(a, \overline{E_{BCp}})=\dist(a, \overline{E_{Cp}})$ for all $a\in A$.
\end{proof}

\begin{theorem}
\label{starindependencetheorem}
The relation $\ind[*]$ is the model theoretic
independence relation for $\rbRT_r$.  Moreover, types over arbitrary sets
of parameters are stationary.
\end{theorem}
\begin{proof}
We will show $\ind[*]$ satisfies all the properties of a stable independence relation
on a universal domain of a stable theory as given in 
Ben Yaacov, Berenstein, Henson and Usvyatsov \cite[Theorem 14.12]{N}.
Then by \cite[Theorem 14.14]{N} we know $\ind[*]$ is the
model theoretic independence relation for the stable theory $\rbRT_r$.

\noindent{\bf (1) Invariance under automorphisms}

Any automorphism $\sigma$ of $\U$ satisfies $\sigma(\overline{E_{Ap}})=\overline{E_{\sigma(Ap)}}$
and is distance preserving.

\noindent{\bf (2) Symmetry:} if $A\ind[*]_C B$, then $B\ind[*]_C A$.

Assume $A\ind[*]_C B$.
This means for all $a\in A$ we have that the closest point in $\overline{E_{BCp}}$
to $a$ is $e_a\in \overline{E_{Cp}}$.
Thus, by Lemma \ref{branch}, for any $a\in A$, for any $y\in \overline{E_{BCp}}$ we have
$[a, y]\cap \overline{E_{Cp}}\not=\emptyset$. It follows that
for any $x\in \overline{E_{Ap}}$, for any $y\in \overline{E_{BCp}}$
there exists a point of $\overline{E_{Cp}}$ on $[x,y]$.
Let $b\in B$. Then for any $x\in \overline{E_{Ap}}$ there is a point of $\overline{E_{p}C}$
on $[x,b]$. It follows that the closest point in $\overline{E_{ACp}}$ to any $b\in B$ is in $\overline{E_{Cp}}$.

\noindent{\bf (3) Transitivity:} $A\ind[*]_C BD$ if and only if $A\ind[*]_C B$ and $A\ind[*]_{BC} D$.

We know
$E_{Cp}\subseteq E_{BCp}\subseteq E_{BCDp}$
which implies
$$\dist(a, \overline{E_{Cp}})\geq \dist(a, \overline{E_{BCp}})\geq \dist(a, \overline{E_{BCDp}}) \text{.} $$
Therefore $\dist(a, \overline{E_{BCDp}})=\dist(a, \overline{E_{Cp}})$ if and only if
$$\dist(a, \overline{E_{BCp}})=\dist(a, \overline{E_{Cp}}) \text{ and }
\dist(a, \overline{E_{BCDp}})=\dist(a, \overline{E_{BCp}}) \text{.} $$
Hence
$$A\ind[*]_C BD \text{ if and only if } A\ind[*]_C B \text{ and } A\ind[*]_{BC} D \text{.} $$

\noindent{\bf (4) Finite character:} $A\ind[*]_C B$ if and only if $a\ind[*]_CB$ for all finite tuples $a\in A$.

This is clear from the definition.

\noindent{\bf (5) Extension:} for all $A, B, C$ we can find $A'$ such that
tp$(A/C)=$ tp$(A'/C)$ and $A'\ind[*]_C B$.

By finite character and saturation of $\U$, it suffices to show this statement when $A$ is
a finite tuple.  Let $e\in \overline{E_{Cp}}$ be the unique point closest to $\overline{E_{Ap}}=E_{Ap}$.
Let $\tau<\kappa$ be the cardinality of $\overline{E_{Bp}}$.  Then there
are at most $\tau$-many branches in $\overline{E_{Bp}}$ at any point of $\overline{E_{Bp}}$.
Since $A$ is finite, use Lemma \ref{embed} to embed a copy of $\overline{E_{Ap}}=E_{Ap}$ on branches
at $e$ that do not intersect $\overline{E_{Bp}}$.
The image of $A$ under this embedding gives us $A'$.

\noindent{\bf (6) Local Character:} if $a=a_1,\dots,a_m$ is a finite tuple,
there is a countable $B_0\subseteq B$ such that $a\ind[*]_{B_0}B$.

Let $e_i$ be the closest point of $\overline{E_{p}B}$ to $a_i$ for $i=1,\dots,m$. Let $B_i$
be a countable subset of $B$ such that $e_i$ is an element of $\overline{E_{B_ip}}$.
Let $B_0=\bigcup_i^m B_i$.

\noindent{\bf  (7) Stationarity (over arbitrary sets of parameters):} 
if $\tp(A/C)=\tp(A'/C)$, $A\ind[*]_C B$, and $A'\ind[*]_C B$,
then
$\tp(A/BC)=\tp(A'/BC)$,
where $C$ is a small submodel of $\U$.

By quantifier elimination, $\tp(A/BC)$ is determined by $\{\tp(a/BC)\mid a\in A\}$ plus
the information $\{d(a_1, a_2)\mid a_1, a_2\in A\}$. These distances
$\{d(a_1, a_2)\mid a_1, a_2\in A\}$ are fixed by $\tp(A/C)$.
Thus, it suffices to show the conclusion in the case when $A=\{a\}$ and $A'=\{a'\}$.
If $a$ or $a'$ is in $Cp$ the conclusion is obvious, so assume $a, a' \notin Cp$.
The type of $a$ (or $a'$) over $BC$ is determined by two parameters,
the unique point in $\overline{E_{BCp}}$ that is closest to $a$, and the distance from $a$
to that point.  Since $a\ind[*]_C B$, it follows that the closest point in $\overline{E_{Cp}}$
to $a$ is the same as the closest point in $\overline{E_{BCp}}$ to $a$,
and the same is true for $a'$.
Since tp$(a/C)=$tp$(a'/C)$, we know $a$ and $a'$ have the same
closest point $e$ in $\overline{E_{Cp}}$ and $d(a, e)=d(a',e)$.  Since $e$ is also the closest point
in $\overline{E_{BCp}}$ to $a$ and $a'$, we know that tp$(a/BC)=$tp$(a'/BC)$
by Lemma \ref{determinetypes}.
\end{proof}

\noindent \textbf{Canonical Bases} \medskip

A canonical base of a stationary type is a minimal set of parameters over which that type is
definable.  However, to avoid a discussion of definable types, we here use an
equivalent definition of canonical base, as given in 
Ben Yaacov, Berenstein, and Henson \cite{IAW}.  
As in that paper, we
here take advantage of the fact (Lemma \ref{L: dcl and acl characterized} and part (7)
of the proof of Theorem \ref{starindependencetheorem}) that every type over an
arbitrary set of parameters is stationary.

For stable theories in general, canonical bases exist as sets of imaginary elements, 
however, in models of $\rbRT_r$, they are sets of ordinary elements.  That is,
the theory has built-in canonical bases.  Indeed, in this setting they are very simple.

For sets $A\subseteq B\subseteq \U$,
and $q\in S_n(A)$ we say $q'\in S_n(B)$ is a \emph{non-forking extension} of $q$ if
$b\models q'$
implies $b\models q$ and $b\ind_A B$.
By the definition of independence, the condition $b\ind_A B$ implies that the points
$e_1,\dots,e_n$ in $\overline{E_{Ap}}$ closest to $b_1,\dots,b_n$ respectively must
also be the closest points to $b_1,\dots,b_n$ in $\overline{E_{Bp}}$.
Because $\rbRT_r$ is stable and all types are stationary, non-forking extensions are unique.
Denote the unique non-forking extension of $q$ to the set $B$ by ${q}{\upharpoonright}^B$. 
Given a type $q$ over a set $A\subseteq \U$ and an automorphism $f$ of $\U$,
$f(q)$ denotes the set of $L_r$-conditions over $f(A)$ corresponding to the
conditions in $q$, where each parameter $a\in A$ is replaced by its image $f(a)$.

The following is Definition 6.1 from Ben Yaacov, Berenstein, and Henson \cite{IAW}.

\begin{definition}
A \emph{canonical base} $Cb(q/A)$ for a type $q\in S_n(A)$ is a subset $C\subseteq \U$ such that
for every automorphism $f\in \Aut(\U)$, we have:
${q}{\upharpoonright}^\U= {f(q)}{\upharpoonright}^\U$ if and only if $f$ fixes each member of $C$.
\end{definition}

The following result describes canonical bases in $\rbRT_r$. 

\begin{theorem}
Let $b=(b_1,\dots,b_n)\in \U^n$ and $A\subseteq \U$ a set of parameters. Let $q \in S_n(A)$
be the type over $A$ of the tuple $b$. Then a canonical base of $q$ is given by the set
$\{e_i \mid 1\leq i\leq n\}$, where $e_i\in \overline{E_{Ap}}$ is the closest point to $b_i$ in $\overline{E_{Ap}}$.
Note that this set depends only on $q$.
\end{theorem}
\begin{proof}
Let $b$,  $A\subseteq \U$ and $q \in S_n(A)$
be as described in the statement of the theorem. Let $C=\{e_i \mid 1\leq i\leq n\}$
where $e_i\in \overline{E_{Ap}}$ is the closest point to $b_i$ in $\overline{E_{Ap}}$.
First, assume $f$ is an automorphism of $\U$ fixing $C$ pointwise.
Let $c=(c_1,\dots,c_n)$ be a realization of ${q}{\upharpoonright}^\U$ (in some extension of $\U$).
Then $c\models q$ and $c\ind_{A} \U$. To show ${f(q)}{\upharpoonright}^\U$=${q}{\upharpoonright}^\U$
it suffices to show that $c\models f(q)$ and $c\ind_{f(A)} \U$, because
then ${q}{\upharpoonright}^\U$ is the unique non-forking extension of $f(q)$ to $\U$.
By Lemma \ref{determinetypes}, an $n$-type over a set $A$ is determined by the values it assigns to the formulas
$d(x_i, x_j)$ and $d(x_i, a)$ for $a\in Ap$.  Note that in $f(q)$, the parameter-free
$L_r$-conditions are the same as in the type $q$. So,
for example, $d(x_i, x_j)$ must have the same value in $f(q)$ as in $q$.  
Thus, to show that $c\models f(q)$ we just need to show
that $d(c_i,a)=d(c_i, f(a))$ for all $a\in Ap$ and $i\in \{1,\dots,n\}$.

We know $c\models q$. Thus Lemma \ref{determinetypes} implies that $e_i$ must be the closest
point to $c_i$ in $\overline{E_{Ap}}$.
Also, $c\ind_A \U$ implies that $e_i$ is also the closest point in $\U$ to $c_i$.
Since $f(\overline{E_{Ap}})\subseteq \U$ we know the closest point in $f(\overline{E_{Ap}})$ to $c_i$ is $e_i=f(e_i)$.
Therefore by Lemmas \ref{branch} and \ref{!geo},
we know $d(c_i, a)=d(c_i, e_i)+d(e_i, a)$ and $d(c_i, f(a))=d(c_i, e_i)+d(e_i, f(a))$ for any $a\in Ap$. Thus, 
\begin{align*}
d(c_i, a)&=d(c_i, e_i)+d(e_i, a)\\
&= d(c_i, e_i)+d(f(e_i), f(a))\\
&=d(c_i, e_i)+d(e_i, f(a))=d(c_i, f(a))
\end{align*}
establishing that $c\models f(q)$.

Since $f$ is an isometry, clearly $f(e_i)=e_i$ is the closest point to $c_i$ in $f(\overline{E_{Ap}})$.
The closed subtree $f(\overline{E_{Ap}})$ is equal to the closed subtree $\overline{E_{f(A)p}}$
since $f([a,b])=[f(a), f(b)]$ for all $a, b\in \U$ and $f(p)=p$.
This implies that $c\ind_{f(A)} \U$.   We conclude that ${f(q)}{\upharpoonright}^\U$=${q}{\upharpoonright}^\U$.

For the other direction, assume $f$ is an automorphism of $\U$
that does not fix all of the elements of $C$. Without loss of generality,
assume $f(e_1)\not=e_1$. Let $(c_1,\dots,c_n)\models {q}{\upharpoonright}^\U$.
Then the closest point in $\U$ to $c_1$ is $e_1$, which is also the closest point to $c_1$ in $\overline{E_{Ap}}$.
Then, since $f(e_1)\in \U$, the point $e_1$ must be
on the geodesic segment joining $f(e_1)$ and $c_1$, so $d(f(e_1), c_1)=d(f(e_1), e_1)+d(e_1,c_1)$.
Thus, $d(f(e_1), e_1)=d(f(e_1), c_1)-d(e_1, c_1)$.
Since $d(f(e_1), e_1)\not=0$, then $d(f(e_1), c_1)\not=d(e_1, c_1)$.
But, $d(e_1, c_1)$ is the value of the formula $d(e_1, x_1)$ in ${q}{\upharpoonright}^\U$,
and by the definition of $f(q)$, the value of $d(f(e_1), x_1)$ in $f(q)$ must equal the value of
$d(e_1, x_1)$ in $q$. So, the $L$-condition $|d(f(e_1), x_1)-d(e_1, c_1)|=0$ is in the type $f(q)$,
and therefore in the type ${f(q)}{\upharpoonright}^\U$.
Thus, the tuple $c=(c_1,\dots,c_n)$ cannot be a realization of ${f(q)}{\upharpoonright}^\U$, and
therefore ${q}{\upharpoonright}^\U\not= {f(q)}{\upharpoonright}^\U$.
\end{proof}





\section{Models of \texorpdfstring{$\rbRT_r$}{rbRTr}: Examples}

In this section we discuss examples of models of $\rbRT_r$ from the literature.  Our first examples
come from the explicitly described universal $\R$-trees that are treated in 
Dyubina and Polterovich \cite{DP2}.  
We show that they give exactly the (fully) saturated models of $\rbRT_r$.  Our second examples come from
asymptotic cones of hyperbolic finitely generated groups.  They give exactly the unique saturated
model of $\rbRT_r$ of density $2^\omega$.

We begin with a lemma about the density of a $\kappa$-saturated model.

\begin{lemma}
\label{kappabranch}
\label{saturateddensity}
Let $\M\models \rbRT_r$ with underlying $\R$-tree $(M,d,p)$, and let $\kappa$
be an infinite cardinal.
\newline
(1) If there exists $a\in M$ with degree $\kappa$, then
the density character of $M$ is at least $\kappa$.
\newline
(2) If  $\M$ is $\kappa$-saturated, then the density character of $\M$ 
is at least $\kappa^\omega$.
\end{lemma}
\begin{proof}
(1) If $d(p,a)=r$, then there is a single branch at $a$, namely the branch
containing $p$. So, we must have $d(p, a)<r$.
Using Lemma \ref{branches of sufficient height} and taking points on different branches at $a$
each with distance $r$ from $p$, we find a collection of
$\kappa$-many points such that the distance between any two of them is $2(r-d(p,a))$.
Thus, the density character of $M$ must be at least $\kappa$.

(2) By Theorem \ref{kappasaturated=kappabranching}, at each point in the 
underlying $\R$-tree of $M$ of $\M$ there are at least $\kappa$-many branches.
The tree construction from Theorem \ref{rbrtstability} then yields at least
$\kappa^\omega$-many distinct points with pairwise distances at least $r$.
\end{proof}

\noindent \textbf{Universal $\R$-trees} \medskip

Next, we review the description of the (isometrically) universal $\R$-trees from 
Dyubina and Polterovich \cite{DP2} and relate them to saturated models of $\rbRT_r$.

\begin{definition}
Let $\mu$ be a cardinal. An $\R$-tree $M$ is called {\it $\mu$-universal}
if, for any $\R$-tree $N$ with $\leq \mu$ branches at every point, there
is an isometric embedding of $N$ into $M$.
\end{definition}

\begin{example}[{\cite[Example 1.1.1 and Lemma 2.1.1]{DP2}}]
\label{universalrtrees}
For each $\mu \geq 2$ let $C_{\mu}$ be
a set with cardinality $\mu$ if $\mu$ is infinite, and cardinality
$\mu-1$ if $\mu$ is finite. Let $A_\mu$ be the set of functions $f\colon (-\infty, \rho_f)\to C_{\mu}$
from an arbitrary left-infinite open interval to $C_{\mu}$, where $f$ satisfies:
\begin{enumerate}
\item There exists $\tau_f\leq \rho_f$ so that $f=0$ for all $t\in (-\infty, \tau_f)$, and
\item The function $f$ is piecewise constant from the right. That is, for
any $t\in (-\infty, \rho_f)$ there exists $\delta>0$ so that $f$ is constant on $[t, t+\delta]$.
\end{enumerate}
On $A_\mu$ define the metric $d$ for distinct $f,g$ by
$d(f, g)=(\rho_f-s)+(\rho_g-s)$ where $s=\sup\{t\mid f(t')=g(t') \, \forall t'<t\}$.

Dyubina and Polterovich \cite{DP2} show that $A_\mu$ is a complete $\R$-tree with $\mu$-many 
branches of infinite height at every point. They also show that $A_\mu$ is the unique 
(up to isometry) $\R$-tree with $\mu$ branches at each point and that it is homogeneous and $\mu$-universal.
In our terminology, for $\mu\geq 3$ the space $A_\mu$ is a complete, unbounded richly branching $\R$-tree.

Let $\M_\mu$ be the pointed $\R$-tree obtained by choosing an arbitrary point $p \in A_\mu$ to be the base point and taking 
$\M_\mu$ to be the closed ball in $A_\mu$ of radius $r$ centered at $p$.
We see that $\M_\mu$ is a model of $\rbRT_r$.  When $\mu$
is infinite, $\M_\mu$ is $\mu$-saturated by Theorem \ref{kappasaturated=kappabranching}
and $\M_\mu$ has density at least $\mu^\omega$ by Lemma \ref{saturateddensity}(2).

Classical model theory suggests that since $\rbRT_r$ is complete and $\kappa$-stable
exactly when $\kappa=\kappa^\omega$, there should be a $\kappa$-saturated model
of $\rbRT_r$ with density $\kappa$ exactly when $\kappa=\kappa^\omega$.  Lemma \ref{saturateddensity}(2)
shows that if $\rbRT_r$ has a $\kappa$-saturated model of density $\kappa$, then we must have $\kappa = \kappa^\omega$.
We verify here that when $\kappa = \kappa^\omega$ and we apply the construction described in the preceding
paragraph to $\mu = \kappa$, we do get the unique $\kappa$-saturated model $\M_\kappa$ of $\rbRT_r$ of density $\kappa$.

To complete the verification, we need to show that when $\mu$ is infinite, the density of $\M_\mu$ is at most $\mu^\omega$.
In fact, an examination of the construction shows that the cardinality of $A_\mu$ is
at most $\mu^\omega$. Conditions (1) and (2) in Example \ref{universalrtrees}
imply that each $f\in A_{\mu}$ can only change values at a countable or finite number of points.
Indeed, given $f \in A_\mu$, we can inductively build a sequence $\{t_n\}_{n\in \N}$
recording where $f$ changes value as follows. For $f\in A_{\mu}$, let $t_0=\tau_f$.

Given $t_n$, let $t_{n+1}=\sup\{t\in (-\infty, \rho_f)\mid f(t)=f(t_n)\}$
if this supremum is $<\rho_f$. If $\sup\{t\in (-\infty, \rho_f)\mid f(t)=f(t_n)\}=\rho_f$,
then set $t_{n+i}=\rho_f$ for all $i\in \N$.
\begin{itemize}
\item if $t_n=t_{n+1}$, then $t_{n+i}=t_n=\rho_f$ for all $i\in \N$;
\item $f$ is constant on $[t_n, t_{n+1})$ for each $n\in \N$.
\end{itemize}
For each $n\in \N$, let $\alpha_n\in C_{\mu}$
be the value of $f$ on $[t_n, t_{n+1})$.
This gives us a sequence $\{\alpha_n\}$ of elements of $C_{\mu}$.
The sequences $\{t_n\}$ and  $\{\alpha_n\}$ determine $f$.
Thus, the cardinality of $A_{\mu}$ is at most $2^\omega \cdot \mu^\omega=\mu^\omega$.

It follows that the metric density of $\M_{\mu}$ is exactly $\mu^\omega$.
Thus, in the case that $\mu=\mu^\omega$, the model $\M_{\mu}$ is a
$\mu$-saturated model of $\rbRT_r$ of density $\mu$, and it is the unique model
with those properties.

In the case that $\mu=\mu^\omega$, we can get an alternative
model-theoretic argument that $A_{\mu}$ is the unique complete $\R$-tree
with $\mu$ branches at every point. 
Given any two complete $\R$-trees $M_1$ and $M_2$ with $\mu$-many branches every point,
select a base point in each.
Then for each $r>0$ the closed $r$-balls in $M_1$ and $M_2$ (centered at their 
respective base points) are saturated models of $\rbRT_r$.   Hence those $r$-balls are isomorphic
by the fact that saturated models of a complete theory with the same density are isomorphic.
A back-and-forth argument can be used to build an isomorphism from $M$ to $N$, where
each time we extend the partial isomorphism to a new point we take its distance from the base point
into account and work in closed $r$-balls for large enough $r$.
\end{example}

\noindent \textbf{Asymptotic Cones} \medskip

A finitely generated group is $\emph{hyperbolic}$ if its Cayley graph is a $\delta$-hyperbolic
metric space for some $\delta>0$. A \emph{non-elementary} hyperbolic group is one 
that has no cyclic subgroup of finite index.

\begin{definition}
Let $(M, d, p)$ be a metric space. Let $U$ be a non-principal ultrafilter on $\N$ and let
$(\nu_m)_{m\in \N}$ be a sequence of positive integers
such that $\lim_{m\to \infty} \nu_m=\infty$.
The asymptotic cone of $(M, d, p)$ with respect to $(\nu_m)_{m\in \N}$ and $U$ is
the ultraproduct of pointed metric spaces $\prod_U (M, {d}/{\nu_m}, p)$. 
Denote this asymptotic cone by $\mathrm{Con}_{U, (\nu_m)}(M, d,p)$.  Elements of
$\mathrm{Con}_{U, (\nu_m)}(M, d,p)$ are denoted $[a_n]$ where $a_n\in M$ for each $n$.
\end{definition}

There are versions of this definition that allow, for example, a different
choice of base point in each factor. Keeping the same base point is sufficient
for our discussion.

\begin{example}
An asymptotic cone $\mathrm{Con}_{U, (\nu_m)}(G)$ of a finitely generated
group is defined to be the asymptotic cone of its Cayley graph with base 
point $e$ and some designated word metric on $G$.
It is a fact that any asymptotic cone of a hyperbolic group is an $\R$-tree and 
is homogeneous (see 
van den Dries and Wilkie \cite{DW} or 
Drutu \cite{D}).
In fact, in the case of a non-elementary hyperbolic group, all asymptotic cones 
are homogeneous with $2^\omega$ branches at every point (see \cite[Proposition 3.A.7]{D}) 
and are thus are isometric to $A_{2^\omega}$ from Example \ref{universalrtrees}.   
A proof of this fact is given below (Lemma \ref{L: hyperbolic group gives R-tree}).
\end{example}

\begin{fact}
\label{switch gen sets}
Say $B$ and $C$ are both finite generating sets for the hyperbolic 
group $G$ and $U$ is a non-principal ultrafilter on $\N$.
The word metrics $d_B$ and $d_C$ are Lipschitz equivalent (and the corresponding
Cayley graphs are quasi-isometric). It follows that the asymptotic cones 
$\mathrm{Con}_{U, (\nu_m)}(G, d_B, e)$ and $\mathrm{Con}_{U, (\nu_m)}(G, d_C, e)$ are homeomorphic. 
\end{fact}

\begin{lemma}
\label{L: hyperbolic group gives R-tree}
Let $G$ be a non-elementary hyperbolic group. Let $U$ be a non-principal ultrafilter on $\N$ and let
$\{\nu_m\}_{m\in \N}$ be a sequence of positive integers such that $\lim_{m\to \infty} \nu_m=\infty$.
Then for any finite generating set $C$ for $G$, the asymptotic cone 
$\mathrm{Con}_{U, (\nu_m)}(G, d_C, e)$  is a homogeneous richly branching $\R$-tree with 
$2^\omega$-many branches at every point.
\end{lemma}
\begin{proof}
Since $G$ is finitely generated, we know that $G$ is countable.
Therefore, any asymptotic cone of $G$ has cardinality (and therefore density)
at most $2^\omega$. By Lemma \ref{kappabranch} we conclude that any point in the cone
can have at most $2^\omega$ branches.
We also know $\mathrm{Con}_{U, (\nu_m)}(G, d_C, e)$ is a homogeneous $\R$-tree.
Since it is non-elementary $G$ contains a free subgroup $F$ with 2 generators (See 
Bridson and Haefliger \cite{BH}).
By Fact \ref{switch gen sets} we may assume without loss of generality that
the generators of $F$ are also generators
of $G$ and that $C$ is a minimal set of generators. In $F$,
for each $m\in \N$, we can find a finite set $C_m$ such that $d(e,a)\geq \nu_m$ for any $a\in C_m$
and so that $(a\cdot b)_e\leq \sqrt{\nu_m}$ for distinct $a, b\in C_m$.
The Cayley graph of $F$ is a subgraph of the Cayley graph of $G$, and the Cayley graph 
of a free group is an $\R$-tree, thus we use $\R$-tree terminology to describe how to find $C_m$.

Given $m\in \N^{>0}$ let $n_m$ be the largest integer so that $n_m\leq \sqrt{\nu_m}$.
Note that $n_m\to\infty$. There are $4\cdot3^{{n_m}-1}$ elements of $F$ with distance $n_m$ 
from $e$. Let $c$ denote such an element. There are 3 branches at $c$ 
in the Cayley graph of $F$ that do not contain $e$.
On each of these branches, choose a point $a$ with $d(a, e)\geq \nu_m$.
Repeat this process for each $c$ with distance $n_m$ from $e$.
Let $C_m$ be the collection of all the points $a$. Then $|C_m|=4\cdot3^{n_m}$, and since $n_m\to \infty$
we know $|C_m|\to\infty$. For any distinct $a, b\in C_m$, the distance from $e$ to $[a,b]$
is at most $n_m\leq \sqrt{\nu_m}$, because $[a,b]$ will always contain at least one of the points
in $F$ with distance $n_m$ from $e$.

Since $|C_m|\to \infty$ and $U$ is a non-principal ultrafilter on $\N$,
we know that $\displaystyle \Pi_m C_m/U$ is a set of points in the asymptotic cone with cardinality $2^\omega$.
Moreover, in the asymptotic cone we have
$$d([a_m],[e])=\lim_U \frac{d_C(a_m, e)}{\nu_m}\geq 1 \text{.}$$
Since $(a_m\cdot b_m)_e\leq \sqrt{\nu_m}$, we know the distance from
 $[e]$ to the geodesic segment connecting $[a_m]$ and $[b_m]$ is
$$([a_m]\cdot [b_m])_{[e]}=\lim_U \frac{(a_m\cdot b_m)_e}{\nu_m}=0 \text{.} $$
Thus, $[e]\in \left[[a_m],[b_m]\right]$, putting $[a_m]$ and $[b_m]$ on separate branches
at $[e]$.  This gives us $2^\omega$-many distinct
branches at $[e]$ in $\mathrm{Con}_{U, (\nu_m)}(G, d_C, e)$. Therefore,
there must be $2^\omega$-many branches at every point in the cone.
\end{proof}

\begin{corollary}
Let $G$ be a non-elementary hyperbolic group.
Let $\M$ be the model of $\rbRT_r$ with underlying space
equal to the closed $r$-ball of  $\mathrm{Con}_{U, (\nu_m)}(G, d_C, e)$.
Then $\M$ is the unique saturated model of $\rbRT_r$ of density $2^\omega$.
\end{corollary}
\begin{proof}
By the preceding lemma and Lemma \ref{kappabranch}, we know $\M$ has density 
$2^\omega$ and $\M$ is $2^\omega$-saturated by Lemma \ref{kappasaturated=kappabranching}.
\end{proof}

\section{Models of \texorpdfstring{$\rbRT_r$}{rbRTr}: Constructions and non-categoricity}
\label{noncat}

In this section we show that $\rbRT_r$ has the maximum number of
models of density character $\kappa$ for every infinite cardinal $\kappa$.
Indeed, for each $\kappa$ we construct a family of $2^\kappa$-many such models such that
no two members of the family are  homeomorphic.  (Two models 
of $\rbRT_r$ are homeomorphic if their underlying $\R$-trees are homeomorphic 
by a map that takes base point to base point.  Note that non-homeomorphic models
of $\rbRT_r$ are necessarily non-isomorphic.)  First we treat 
separable models, and the amalgamation techniques used
in that case also allow us to characterize the principal types of $\rbRT_r$ and
to show that this theory has no atomic model.  Then we use simple amalgamation
constructions to handle nonseparable models.

\begin{lemma}
\label{separabletreewitharbbranching}
Let $S$ be a non-empty set of integers, each of which is $\geq 3$.
There exists a separable richly branching $\R$-tree $M$ such that
\begin{enumerate}
\item for each $k\in S$ the set $\{x\in M\mid x \text{ has degree }k\}$ is dense in $M$, and
\item given a branch point $x\in M$ the degree of $x$ is an element of $S$.
\end{enumerate}
\end{lemma}
\begin{proof}
Let $(k_j\mid j\in \N)$
be a sequence such that every element of $S$ appears infinitely many times in the sequence,
and every term of the sequence is an element of $S$.
We construct an increasing sequence $N_0\subseteq N_1\subseteq \dots \subseteq N_j\dots $ of 
separable $\R$-trees as follows.

Let $N_0$ be the $\R$-tree $\R$ with base point $0$.
Let $A_0$ be a countable, dense subset of $N_0$. 
Use Lemma \ref{chiswell add to points} to add $k_0-2$ distinct rays (copies of $\R^{\geq 0}$) at each
point in $A_0$, bringing the number of branches of infinite length at each point in $A_0$ up to $k_0$.
Call the resulting $\R$-tree $N_1$. Note that $N_0\subseteq N_1$.
The $\R$-tree $N_1$ is separable, since it is a countable union of separable spaces.
Note also that all the points in $N_1\setminus A_0$ only have 2 branches,
and it is straightforward to show $N_1\setminus A_0$ is uncountable and dense in $N_1$.

Once $N_j$ has been constructed, to construct $N_{j+1}$
let $\displaystyle A_{j}\subset N_j\setminus (\cup_{i=0}^{j-1} A_j)$ be a countable, dense subset of $N_j$. 
This is possible since $\displaystyle  N_j\setminus (\cup_{i=0}^{j-1} A_j)$ is dense in $N_j$.
Use Lemma \ref{chiswell add to points} to add $k_j-2$ rays at each
point in $A_j$, bringing the number of branches of infinite length at each point in $A_j$ up to $k_j$.
The resulting $\R$-tree is $N_{j+1}$.
Note that $N_{j+1}$ is separable, since it is a countable union of separable spaces.
Note also that all the points in $\displaystyle N_{j+1}\setminus \cup_{j=0}^{j} A_j$ still only have 2 branches,
and that this set is uncountable and dense in $N_{j+1}$.
Lastly, it is clear that given $x\in N_{j+1}$ the number of branches at $x$ must be
either $2$ (in which case $x$ is not a ``branch point") or one of $\{k_0, \dots, k_j\}$.
This is because at the $j$th step, the only points at which we add rays are those in $A_j$,
and then in subsequent steps we do not add rays at any of those points.

Let $\displaystyle M=\cup_{j\in\N} N_j$ be the union of this countable chain of separable 
$\R$-trees. Then $M$ is a separable $\R$-tree (see Chiswell \cite[Lemma 2.1.14]{C}.)
Since for each $N_j$ the number of branches at each branch point is an element
of $S$, this will also be true in $M$.
Let $k\in S$. Let $\displaystyle J(k)=\{j\in \N\mid k_j=k\}$. By how we chose the sequence $(k_j)$ the set $J(k)$
is infinite.  The set of points in $M$ which have exactly $k$ branches is $\cup_{j\in J(k)} A_j$.
We will show this set is dense in $M$.
Let $x\in M$. Let $j_x\in \N$ be the smallest integer such that $x\in N_{j_x}$. Let $j^*\in J(k)$ be
such that $j^*>j_x$. We know that $A_{j^*}$ is dense in $N_{j^*}$, and that $\displaystyle x\in N_{j_x}\subseteq N_{j^*}$.
Therefore, there are points in $\displaystyle A_{j^*}\subseteq  \cup_{j\in J(k)} A_j$ arbitrarily close to $x$.
Our choice of $x\in M$ was arbitrary, therefore $\displaystyle \cup_{j\in J(k)} A_j$ is dense in $M$.
Because we chose a non-empty $S$ with members all $\geq 3$
the set of branch points with at least 3 branches of infinite length is dense in $M$. Therefore, $M$
is a richly branching $\R$-tree.
\end{proof}

\begin{remark}
In the preceding proof, we did not use the fact that we are considering homeomorphisms
of \emph{pointed} topological spaces. The $\R$-trees constructed in Lemma \ref{separabletreewitharbbranching}
are in fact non-homeomorphic, even when we are not required to preserve the base point.
\end{remark}

\begin{theorem}  
\label{number of separable models}
There exist $2^{\omega}$-many pairwise non-homeomorphic
(hence non-isomorphic) separable models of $\rbRT_r$.
\end{theorem}
\begin{proof}
Any homeomorphism $g$ between models $\M$ and $\n$ of $\rbRT_r$ is a homeomorphism
on the underlying $\R$-trees which must preserve branching.  In particular, given $n\in \N^{\geq 3}$,
if there is a point with degree $n$ in $M$, then there must be a point with degree $n$ in $N$.
Choose two different subsets $S$ and $S'$ of $\N^{\geq 3}$, and construct
a richly branching tree for each as in Lemma \ref{separabletreewitharbbranching}. Let
$\M$ and $\M'$ be the models of $\rbRT_r$ based on the completions of their closed $r$-balls, respectively.
By Lemma \ref{completing a tree}, taking the completion only adds points of degree 1.
It follows that $\M$ and $\M'$ cannot be homeomorphic.
Since there are $2^\omega$-many different such sets $S$,
there are $2^\omega$-many different non-homeomorphic, separable models of $\rbRT_r$.
\end{proof}

Recall that given a continuous theory $T$, a type $q\in S_n(T)$ is \emph{principal}
if for every model $\M$ of $T$, the set $q(\M)$ of realizations of $q$ in $\M$ is definable
over the empty set.  As in classical first order logic, given a complete theory in a countable signature, there is
a Engeler-Ryll-Nardzewski-Svenonius theorem stating the equivalence between $\omega$-categoricity and the fact
that every type is principal.  (See 
Ben Yaacov, Berenstein, Henson and Usvyatsov \cite[Theorem 12.2]{N}.)   
Furthermore, a type $q$ is principal
if and only if $q$ is realized in every model of $T$.  (See \cite[Theorem 12.6]{N}.)

Theorem \ref{number of separable models} obviously implies that $\rbRT_r$ is not $\omega$-categorical, 
and thus not every type of $\rbRT_r$ is principal.   Our next result gives a characterization of the 
principal  types in $S_n(\rbRT_r)$. In particular,
a principal type is the type of a tuple of points that all lie along a single piecewise segment
with $p$ as an endpoint. Thus, there are very few of them. 
As a consequence, we conclude that $\rbRT_r$ does not have a prime model 
(equivalently, does not have an atomic model, one in which only principal types are realized).

For a clear and comprehensive treatment of separable models in continuous model theory,
we refer the reader to Section 1 in 
Ben Yaacov and Usvyatsov \cite{BUdfinite}.  
Note that where we and 
Ben Yaacov, Berenstein, Henson and Usvyatsov \cite{N} use the 
word \emph{principal}, the authors of \cite{BUdfinite} use \emph{isolated,} which is now the standard terminology.
In (\cite[Theorem 1.11]{BUdfinite}) they prove an omitting types theorem, and as a 
corollary (\cite[Corollary 1.13]{BUdfinite}) show that nonprincipal types can be omitted.
Further, it follows from \cite[Definition 1.7]{BUdfinite} and properties of definable sets in 
continuous model theory, that every principal type is realized in every model, and this is implicit 
in the discussion following that definition.

\begin{theorem}
\label{principal types}
Let $q\in S_n(\rbRT_r)$. The following are equivalent
\begin{enumerate}
\item The type $q$ is principal.
\item There exists a permutation $\sigma$ of $\{1,\dots,n\}$ such that $q$ contains
the condition 
\[
d(p, x_{\sigma(n)})=d(p, x_{\sigma(1)})+\sum_{i=1}^{n-1} d(x_{\sigma(i)}, x_{\sigma(i+1)}) \text{.}
\]
\item  There exists a permutation $\sigma$ of $\{1,\dots,n\}$ such that 
for any model $\M$ of $\rbRT_r$ and $(b_1,\dots, b_n)\in M^n$ that realizes $q$ in $\M$,  
we have that $[p,b_{\sigma(1)},\dots,b_{\sigma(n)}]$ is a piecewise segment in $\M$.
\end{enumerate}
\end{theorem}
\begin{proof}
That (2) and (3) are equivalent (for the same permutation $\sigma$) 
follows from Chiswell \cite[Lemma 2.1.4]{C}.  See the discussion
after Definition \ref{piecewise segment}.

We show (1) implies (3) by proving the contrapositive.
Let $\M\models \rbRT_r$ and suppose $(b_1, \dots, b_n)\in M^n$ realizes $q$ in $\M$.
For each $i$ let $t_i$ be the number $d(p,b_i)$.
Choose $j \in \{1,\dots,n\}$ such that $d(p,b_j) \geq d(p,b_i)$ holds for every
$i \in \{1,\dots,n\}$.  If (3) fails, there must be $i$ such that $b_i$ is not
on the segment $[p,b_j]$.  Note that this implies that $p,b_i,b_j$ are distinct,
so $t_i$ and $t_j$ are $>0$.  (If $d(p,p_j)=0$ then $d(p,b_i)=0$ for all $i$, meaning
that $b_1 = \dots b_n = p$, contradicting the assumption that (3) fails.)

Let $e=Y(p,b_i,b_j)$ be the closest point to $b_i$ on the segment $[p, b_j]$. If $e=b_j$, then
$$d(p, b_i)=d(p, e)+d(e, b_i)=d(p, b_j)+d(b_j, b_i)>d(p, b_j) \text{,}$$
which is a contradiction.
This leaves two possibilities.

Case 1: Assume $e=p$. Then $b_i$ and $b_j$ are on different branches at $p$, so
$q$ contains the conditions $d(x_i,x_j) = d(p,x_i)+d(p,x_j)$, $d(p,x_i)=t_i$ and
$d(p,x_j)=t_j$.  It follows that whenever $(c_1,\dots,c_n)$ realizes $q$ in any model
$\n$ of $\rbRT_r$, the points $c_i$ and $c_j$ must be on different branches at $p$.
However, using techniques as in
\ref{separabletreewitharbbranching}, we can construct $\n\models \rbRT_r$ in which there
is only a single branch at $p$. Then $q$ is not realized in $\n$ and thus $q$ is not principal.

Case 2: Assume $e\not=p$.
Then $p$ and $b_j$ are on different branches at $e$,
and Lemma \ref{!geo} implies that $e\in [p, b_i]$ and $e\in [b_i, b_j]$. 
So $p, b_i, b_j$, being 3 distinct points, are on 3 distinct branches at
$e = Y(p,b_i,b_j)$.  As discussed in the preceding case, these properties
of $p,b_i,b_j$ are witnessed by conditions in $q$; also, $q$ contains the
condition $d(p,Y(p,x_i,x_j)) = t$ for some $t \in (0,r]$.  It follows that in
any model of $\rbRT_r$ in which $q$ is realized, there must be a point $c$
whose distance from $p$ is $t$ and at which the model has 3 branches.
Using techniques discussed earlier in this section, we may 
build a model $\n$ in which there are no branch points at distance
$t$ from $p$. Then $q$ is not realized in $\n$, implying that $q$ is not
a principal type.

Last, we show (3) implies (1).
Let $q\in S_n(\rbRT_r)$ be a type satisfying the conditions in (3) (and hence in (2)) with permutation $\sigma$.
Then $d(p, x_{\sigma(n)})^q\geq d(p, x_{\sigma(i)})^q$ for all $i\in \{1,..,n\}$.
Take any model $\M\models \rbRT_r$. Then there is at least one branch
at the base point $p$, and by Lemma \ref{branches of sufficient height} along this branch 
we can find a point $b$ so that $d(p, b)=d(p, x_{\sigma(n)})^q$. On the segment $[p,b]$ we
can find points $b_1,\dots,b_n$ such that $d(p,b_{\sigma(i)}) = d(p, x_{\sigma(i)})^q$ for
all $i$.  From this it follows that $(b_1,\dots,b_n)$ realizes $q$ in $\M$.  Since $q$ is
realized in every model of $\rbRT_r$, it is a principal type.
\end{proof}

\begin{corollary}
The $L$-theory $\rbRT_r$ has no prime model.
\end{corollary}
\begin{proof}
Assume $\rbRT_r$ has a prime model $\M$ with underlying $\R$-tree
$(M, d,p)$. Then $\M$ is atomic and the type
of any tuple $b_1,\dots,b_n$ must be principal. By the preceding theorem, this means
that for any pair $b_1, b_2$ in $M$, either $b_1\in [p, b_2]$ or $b_2\in [p, b_1]$.
It follows that $M$ consists of a piecewise segment with endpoint $p$.  In particular, 
$\M$ is not richly branching, which is a contradiction.
\end{proof}

We finish this section by showing that when $\kappa$ is uncountable,
then the number of different models of $\rbRT_r$ having density character
equal to $\kappa$ is also the maximum possible, namely $2^\kappa$.
As in the case $\kappa = \omega$, which was treated in the first part of this
section, we produce large sets of models that are not only non-isomorphic,
but in fact have underlying $\R$-trees which are non-homeomorphic
(as pointed topological spaces).

We will carry out the construction by induction on $\kappa$, and we begin with a useful lemma. 
\begin{lemma}
\label{getting many nonseparable models}
Let $\kappa$ be an uncountable cardinal.  
The following conditions are equivalent:
\newline
(1) The number of non-homeomorphic models of $\rbRT_r$ of density 
character $\leq \kappa$ is at least $\kappa$.
\newline
(2) The number of non-homeomorphic models of $\rbRT_r$ of density character $\leq \kappa$ 
that have just one branch at the base point is at least $\kappa$.
\newline
(3) The number of non-homeomorphic models of $\rbRT_r$ of density 
character $= \kappa$ is $2^\kappa$.
\end{lemma}
\begin{proof}
Let $\kappa$ be an uncountable cardinal. Clearly, (3) implies (1).
To show (2) implies (3), assume (2) and  let $(\B_\alpha \mid \alpha <  \kappa)$ be 
a list of the pairwise non-homeomorphic models of $\rbRT_r$, each with density 
character $\leq \kappa$ and exactly one branch at the base point.  Given a subset 
$S\subseteq \kappa$ of cardinality $= \kappa$, take the collection of $B_\alpha$
for $\alpha\in S$ and glue them all together at their base points
using Theorem \ref{amalg}.  Call this amalgam $M_S$,
and let $p$ be the point in $M_S$ at which the $B_\alpha$ are all glued together. 
Make $p$ the base point of the $L_r$-structure $\M_S$.   Note that the density character 
of $\M_S$ is exactly $\kappa$, since each branch of 
its underlying $\R$-tree $M_S$ at $p$ has density $\leq \kappa$, 
and there are exactly $\kappa$-many branches at $p$.
Moreover, it is easy to check that $\M_S$ is a model of $\rbRT_r$.

By this construction, if $B$ ranges over the branches of $M_S$ at $p$, the
homeomorphism type of $B \cup \{p\}$ (with $p$ as distinguished element)
ranges bijectively over the homeomorphism types of $B_\alpha$ (also with $p$ as 
distinguished element) as $\alpha$ ranges over $S$.  It follows that the homeomorphism type of $\M_S$ determines $S$.
Therefore the family $\{ \M_S \mid S \subseteq \kappa \text{ and } S \text{ has cardinality } = \kappa \}$
verifies condition (3), since $\kappa$ has $2^\kappa$-many subsets of cardinality $= \kappa$.

Finally, we prove that (1) implies (2).
For each model $\M$ of $\rbRT_r$, let $b(M)$ denote the number of branches of $M$ at 
its base point; we take this to be a positive integer or $\infty$, where $b(M) = \infty$ means 
that there are infinitely many branches.  Since $\kappa$ is uncountable, condition (1) yields 
a class $\mathcal{K}$ of at least $\kappa$-many non-homeomorphic models of $\rbRT_r$, 
each of density character $\leq \kappa$, such that $b(M)$ has a constant value $b$ as $M$ 
ranges over $\mathcal{K}$.  We may assume $b \neq 1$, since otherwise condition (2) is 
satisfied by the models in $\mathcal{K}$.

Let $a$ be an integer $\geq 3$ that is different from $b+1$.  ( Note that $a\not=b+1$ is  automatically 
true if $b$ is $\infty$.)  Using a method similar to that in the proof of Lemma \ref{separabletreewitharbbranching}, 
we may take $\n=(N, d, p)$ to be a separable model of $\rbRT_r$ with the following properties: 
(1) for all $x$ in $N$, the number of branches in $N$ at $x$ is $1$ or $2$ or $a$; 
(2) $N$ has a single branch at its base point $p$; and (3) 
$N$ has a single branch at some point $y$, where $d^{\n}(p, y)={r}/{2}$.
To get single branches at 2 points with a given distance as required here,
proceed as in the proof of Lemma \ref{separabletreewitharbbranching}, 
starting with the interval $[0, {r}/{2}]\subseteq \R$
instead of $\R$, and in subsequent steps always exclude $0$ and ${r}/{2}$ 
from the sets of points where rays are added. 

Now consider an arbitrary $\M \in \mathcal{K}$, and denote the base point of $\M$ by $q$.
Scale the metric on $M$ down by a factor of 2, resulting in an $L_r$-structure with a radius of 
${r}/{2}$.
We construct a larger $\R$-tree $M^*$ by amalgamating the scaled-down $M$ and 
$\n$ in the way that identifies $q$ and $y$; we will denote this point of $M^*$ by $qy$.  
We take the base point of $\M^*$ to be the base point $p$ of $\n$.  The radius of  $\M^*$ 
is $r$, because the radius of $\n$ was $r$, and by our amalgamation construction, for 
every point $x\in M^*$ residing in the scaled down copy of $M$, 
$$d^{\M^*}(p,x)=d^{\M^*}(p, qy)+d^{\M^*}(qy, x)=\frac{r}{2}+\frac{d^{\M}(y, x)}{2}$$
which attains its maximum value $r$ as $x$ ranges over the scaled down copy of $M$.

It is straightforward to check that $\M^*$ is a model of $\rbRT_r$, has density $\leq \kappa$, 
and has a single branch at its base point.  Note that the branches of $M^*$ at the amalgamated 
point $qy$ consist of the branches of $q$ in $M$ together with the tree that results from 
$N$ by removing $y$.  In particular, this means that $M^*$ has $b+1$-many branches at $qy$.

We claim that the class $\mathcal{K}^* = \{ \M^* \mid \M \in \mathcal{K} \}$ verifies condition (3); 
it remains only to show that no two members of this class are homeomorphic.  (Recall that 
the homeomorphisms we consider must take base point to base point.)  The key to this is the 
fact that the point $qy$ can be topologically identified in $M^*$, given that we know the base 
point $p$.  To do this, note first that the segment $X=[p,qy)$ in $M^*$ is identical to the segment 
$[p,y)$ in $N$, and every point in $X$ has the same number of branches in $M^*$ as in 
$N$. Therefore every point $x$ of $X$ has  $1$, $2$, or $a$-many branches in $M^*$, and 
thus the number of branches at $x$ is different from the number of branches at $qy$.  From 
this we conclude that for any $\M_1,\M_2 \in \mathcal{K}$, any homeomorphism of $M_1^*$ 
onto $M_2^*$ that takes base point to base point must map the scaled version of $M_1$ 
onto the scaled version of $M_2$.  Since this can only happen when $\M_1 = \M_2$, by 
assumption on $\mathcal{K}$, we conclude that $\M_1^* = \M_2^*$, as desired.
\end{proof}

\begin{theorem}
Let $\kappa$ be an uncountable cardinal. The number of non-homeomorphic models of 
$\rbRT_r$ of density character equal to $\kappa$ is $2^\kappa$.
\end{theorem}
\begin{proof}
We assume that $\sigma$ is the least uncountable cardinal at which there are strictly 
fewer than $2^\sigma$-many non-homeomorphic models of density character equal to 
$\sigma$, and derive a contradiction.  Using Theorem \ref{number of separable models}, 
we see that condition (1) in Theorem \ref{getting many nonseparable models} holds when 
$\kappa = \omega_1$; condition (3) in that result yields that $\rbRT_r$ has 
$2^{\omega_1}$-many non-homeomorphic models of density character equal to $\omega_1$.  Thus 
$\sigma > \omega_1$.  Now suppose $\sigma$ is a successor cardinal; say it is the next 
cardinal bigger than $\lambda$, which must be uncountable.  Our choice of $\sigma$ ensures 
that there must be $2^\lambda \geq \sigma$-many non-homeomorphic models of density 
character $\lambda$.  Applying Lemma \ref{getting many nonseparable models} with 
$\kappa = \sigma$ gives a contradiction; indeed, we have verified condition (1), while condition (3) is false.
So $\sigma$ must be a limit cardinal.  Let $\tau$ be the number of non-homeomorphic 
models of $\rbRT_r$ that have density character $\leq \sigma$; our treatment of 
$\omega_1$ shows that $\tau$ is uncountable.  Furthermore, Lemma \ref{getting many 
nonseparable models} applied to $\kappa = \sigma$ yields $\tau < \sigma$.   Our choice 
of $\sigma$ ensures that there are $2^ \tau > \tau$-many non-homeomorphic models of 
$\rbRT_r$ that have density character $\tau$, contradicting the definition of $\tau$.
\end{proof}

\section{Unbounded \texorpdfstring{$\R$}{R}-trees}

As noted in the Introduction, we have chosen to treat \emph{bounded} pointed $\R$-trees
in this paper, because many of the model-theoretic ideas and tools we need from continuous
first order logic are only documented for bounded metric structures in the literature.

However, it would certainly be a natural research topic to study the model theory of unbounded
({ie}, not necessarily bounded) pointed $\R$-trees.  The most immediately available setting
for doing this would be to consider a pointed $\R$-tree $(M,p)$ as a many-sorted metric structure
in which each sort is one of the (closed) bounded balls of $(M,p)$ (centered at $p$), and the union
of the family of distinguished balls is all of $M$.  
Everything done in this paper can easily be carried over to that setting.  The disadvantages of
doing so are the technical awkwardness of the many-sorted framework and the need for imposing
an arbitrary family of radii for the bounded balls into which the full tree is stratified.

It is certainly more mathematically natural to consider pointed $\R$-trees on their own, without
imposing a many-sorted stratification.  There are suitable
logics for doing model theory with such unbounded structures.  For example, a version of continuous first
order logic for unbounded metric structures is described in Ben Yaacov \cite{BY}.  Also, a logic based
on \emph{positive bounded formulas} and an associated concept of \emph{approximate satisfaction}
is presented in Section 6 of Due\~{n}ez and Iovino \cite{DI}.  However, for neither of these approaches 
are the ideas and tools of model theory developed as we need them in this paper.

In each of these three available settings for treating arbitrary pointed $\R$-trees, the arguments in this paper
can be used easily to demonstrate: (1) the class of pointed $\R$-trees is axiomatizable and (2) for each $r>0$, the 
ball $\{ x \mid d(x,p) \leq r \}$ is a definable set (over $\emptyset$, uniformly in all pointed $\R$-trees).  
Together with what is developed in 
Ben Yaacov, Berenstein, Henson and Usvyatsov \cite{N}
as well as in Ben Yaacov \cite{BY} and in
Due\~{n}ez and Iovino \cite{DI}, this quickly yields that the model theoretic 
frameworks for pointed $\R$-trees provided by these three settings are completely equivalent.  
In particular, this approach yields a model companion 
for the theory of pointed $\R$-trees whose models are exactly the
richly branching $\R$-trees ({ie}, the complete pointed $\R$-trees described in 
Remark \ref{general richly branching}).  Furthermore, this
model companion has suitably stated versions of all the properties of $\rbRT_r$ that are proved in this paper.





\end{document}